\documentclass[12pt]{article}
\usepackage[colorlinks=true,linkcolor=blue,citecolor=Red]{hyperref}
\usepackage{amsmath}

\usepackage{mathtools}

\newcommand{\IGPT}
{$\mathcal{I}\mathcal{G}\mathcal{P}$-\text{T}}
\newcommand{\cN}{\mathcal{N}}
\newcommand{\F}{\mathscr{F}_k}
\renewcommand{\epsilon}{\varepsilon}
\newcommand{\IGP}
{\mathcal{I}\mathcal{G}\mathcal{P}}

\usepackage{bbm}

\usepackage{amsmath}
\usepackage{amssymb}
\usepackage{amsthm}
\usepackage{url}
\usepackage{bbm}

\usepackage{color}

\definecolor{Red}{rgb}{1,0,0}
\definecolor{Blue}{rgb}{0,0,1}
\definecolor{Olive}{rgb}{0.41,0.55,0.13}
\definecolor{Yarok}{rgb}{0,0.5,0}
\definecolor{Green}{rgb}{0,1,0}
\definecolor{MGreen}{rgb}{0,0.8,0}
\definecolor{DGreen}{rgb}{0,0.55,0}
\definecolor{Yellow}{rgb}{1,1,0}
\definecolor{Cyan}{rgb}{0,1,1}
\definecolor{Magenta}{rgb}{1,0,1}
\definecolor{Orange}{rgb}{1,.5,0}
\definecolor{Violet}{rgb}{.5,0,.5}
\definecolor{Purple}{rgb}{.75,0,.25}
\definecolor{Brown}{rgb}{.75,.5,.25}
\definecolor{Grey}{rgb}{.5,.5,.5}

\usepackage{url}
\usepackage{algorithm}
\usepackage{algorithmic}
\newcommand{\ind}{\mathbbm{1}}

\renewcommand{\epsilon}{\varepsilon}

\newcommand{\R}{\mathbb{R}}

\newcommand{\ip}[2]{\langle{#1},{#2}\rangle}

\renewcommand{\R}{\mathbb{R}}

\newcommand{\ignore}[1]{\relax}

\newlength\myindent
\newtheorem{theorem}{Theorem}[section]
\newtheorem{remark}[theorem]{Remark}
\newtheorem{lemma}[theorem]{Lemma}

\newtheorem{proposition}[theorem]{Proposition}

\newtheorem{definition}[theorem]{Definition}
\newtheorem{Assumption}[theorem]{Assumption}

\newcommand{\ER}{Erd{\"o}s-R\'{e}nyi }

\renewcommand{\ip}[2]{\left\langle#1,#2\right\rangle}

\usepackage{mathrsfs}


\newcounter{parentnumber}

\makeatletter
\def\BState{\State\hskip-\ALG@thistlm}
\makeatother

\definecolor{Red}{rgb}{1,0,0}
\definecolor{Blue}{rgb}{0,0,1}
\definecolor{Olive}{rgb}{0.41,0.55,0.13}
\definecolor{Green}{rgb}{0,1,0}
\definecolor{MGreen}{rgb}{0,0.8,0}
\definecolor{DGreen}{rgb}{0,0.55,0}
\definecolor{Yellow}{rgb}{1,1,0}
\definecolor{Cyan}{rgb}{0,1,1}
\definecolor{Magenta}{rgb}{1,0,1}
\definecolor{Orange}{rgb}{1,.5,0}
\definecolor{Violet}{rgb}{.5,0,.5}
\definecolor{Purple}{rgb}{.75,0,.25}
\definecolor{Brown}{rgb}{.75,.5,.25}
\definecolor{Grey}{rgb}{.5,.5,.5}
\definecolor{Pink}{rgb}{1,0,1}
\definecolor{DBrown}{rgb}{.5,.34,.16}
\definecolor{Black}{rgb}{0,0,0}

\usepackage{float}
\newcommand{\bs}{\boldsymbol{\sigma}}
\usepackage{float}

\renewcommand{\epsilon}{\varepsilon}

\usepackage{algorithm}
\usepackage{algorithmic}

\title{Large Average Subtensor Problem: Ground-State, Algorithms, and Algorithmic Barriers} 
\author{\sf{Abhishek Hegade K. R.}\thanks{Department of Physics, University of Illinois Urbana-Champaign; e-mail: {\tt ah30@illinois.edu}} \and \sf{Eren C. K{\i}z{\i}lda\u{g}}\thanks{Department of Statistics, University of Illinois Urbana-Champaign; e-mail: 
 {\tt kizildag@illinois.edu}}}
\date{}
\begin{document}
\maketitle
\begin{abstract}
We introduce the large average subtensor problem: given an order-$p$ tensor over $\R^{N\times \cdots \times N}$ with i.i.d.\,standard normal entries and a $k\in\mathbb{N}$, algorithmically find a $k\times \cdots \times k$ subtensor with a large average entry. This generalizes the large average submatrix problem, a key model closely related to biclustering and high-dimensional data analysis, 
to tensors. For the submatrix case, Bhamidi, Dey, and Nobel~\cite{bhamidi2017energy} explicitly highlight the regime $k=\Theta(N)$ as an intriguing open question.

Addressing the regime $k=\Theta(N)$ for tensors, we establish that the largest average entry concentrates around an explicit value  $E_{\mathrm{max}}$, provided that the tensor order $p$ is sufficiently large. Furthermore, we prove that for any $\gamma>0$ and large $p$, this model exhibits multi Overlap Gap Property ($m$-OGP) above the threshold $\gamma E_{\mathrm{max}}$. The $m$-OGP serves as a rigorous barrier for a broad class of algorithms exhibiting input stability. These results hold for both $k=\Theta(N)$ and $k=o(N)$. Moreover, for small $k$, specifically  $k=o(\log^{1.5}N)$, we show that a certain
polynomial-time algorithm identifies a subtensor with  average entry $\frac{2\sqrt{p}}{p+1}E_{\mathrm{max}}$. In particular, the $m$-OGP is asymptotically sharp: onset of the $m$-OGP and the algorithmic threshold match as $p$ grows.

Our results show that while the case $k=\Theta(N)$ remains open for submatrices, it can be rigorously analyzed for tensors in the large $p$ regime. This is achieved by interpreting the model as a Boolean spin glass and drawing on insights from recent advances in the Ising $p$-spin glass model.
\end{abstract}
\newpage
\tableofcontents
\newpage
\section{Introduction}
In this paper, we introduce the \emph{large average subtensor} problem. Given an order-$p$ tensor 
\begin{equation}\label{eq:A_Tensor}
    \boldsymbol{A}=\bigl(A_{i_1,\dots,i_p},1\le i_1,\dots,i_p\le N\bigr)\in(\R^N)^{\otimes p}
\end{equation}
with i.i.d.\,standard normal entries, $A_{i_1,\dots,i_p}\sim \cN(0,1)$, and a positive integer $k$ that may grow with $N$, our objective is to algorithmically find a $k\times \cdots \times k$ subtensor with a large average entry. Formally, this corresponds to solving the following random optimization problem:
\begin{equation}\label{eq:MaxT}
M^*:=\max_{\substack{I_1,\dots,I_p\subset[N] \\ |I_1|=\cdots=|I_p|=k}}\frac{1}{k^p}\sum_{i_1\in I_1,\dots,i_p\in I_p}A_{i_1,\dots,i_p},
\end{equation}
where $[N]:=\{1,\dots,N\}$. In what follows, we refer to $M^*$ as the \emph{ground-state value}. 

The submatrix case ($p=2$) is particularly motivated by applications spanning from genomics and biomedicine to social networks~\cite{madeira2004biclustering,shabalin2009finding,fortunato2010community}, where identifying large average submatrices is commonly referred to as \emph{biclustering}~\cite{madeira2004biclustering}. For instance, in biomedical applications, the rows of the matrix represent measured biological quantities and the columns correspond to samples. 
A submatrix with a large average entry reveals meaningful interactions between corresponding biological quantities and samples, see~\cite{bhamidi2017energy} and references therein for details. 
For the submatrix case, rigorous results, including the ground-state value and algorithmic guarantees, were obtained in~\cite{sun2013maximal},~\cite{bhamidi2017energy}, and~\cite{gamarnik2018finding}, some of which we review in Section~\ref{sec:Submatrix-Results}.

Tensors naturally represent complex, multi-way interactions arising in numerous important applications (see, e.g.,~\cite{nion2010tensor,zhou2013tensor,zhang2019tensor} as well as the surveys~\cite{bi2021tensors,auddy2024tensor}), yet we are unaware of any prior work addressing the large average subtensor problem, i.e., the case $p\ge 3$. This gap represents a critical missing piece in the literature. Even more importantly, existing work on the submatrices ($p=2$) predominantly assumes $k=O(\log N)$, with the extension beyond this regime being unknown. 
Notably,~\cite[Section~2.7.1]{bhamidi2017energy} explicitly highlight the case $k=\alpha N$, $\alpha\in(0,1)$ as a very intriguing open problem that will require new ideas. More recently,  \cite{Erba_2024} put forth a rich phase diagram for the large average submatrix problem when $k=\Theta(N)$, by employing the non-rigorous replica method from statistical mechanics~\cite{1987sgtb.book.....M}; see Section~\ref{sec:Submatrix-Results}. Collectively, these works suggest that the regime $k=\Theta(N)$ harbor very intriguing properties, meriting further investigation. 
However, rigorous results for $k=\Theta(N)$ remain missing both in the submatrix case ($p=2$) as well as in the subtensor case ($p\ge 3)$. This serves as the main motivation of our paper, leading us to ask:
\begin{center}
    \emph{Can rigorous results be established when $k=\Theta(N)$ for the submatrix or the subtensor case?}
\end{center} 
We answer this question affirmatively for subtensors in the regime $N\to\infty$ and $p$ remains a large constant in $N$, which formally corresponds to the double limit $N\to\infty$  followed by $p\to\infty$.  While the submatrix case when $k=\Theta(N)$ still remains open and likely requires advanced tools---such as those from the Sherrington-Kirkpatrick Model~\cite{sherrington1975solvable,talagrand2006parisi,talagrand2011mean,panchenko2013sherrington}---our analysis demonstrates that for the subtensor case, rigorous results for $k=\Theta(N)$ can be established using more accessible techniques, such as refined applications of the first and the second moment methods, provided the order parameter $p$ is large. Focusing on the large parameter regime is common in the study of random structures; see Remark~\ref{remark:Large-p}.

More specifically, we determine the ground-state value for the random optimization problem~\eqref{eq:MaxT} when $p$ is large, both in the challenging regime $k=\Theta(N)$ and in the case $k=o(N)$. We then establish an algorithmic lower bound based on the Overlap Gap Property (OGP), an intricate feature of the optimization landscape that serves as a rigorous barrier against a broad array of algorithms~\cite{Gamarnik_2021,gamarnik2025turing}. Our OGP result also holds in both regimes, $k=\Theta(N)$ and $k=o(N)$. Lastly, focusing on smaller $k$, specifically $k=o(\log^{1.5} N)$, we devise a polynomial-time algorithm that achieves a $2\sqrt{p}/(p+1)$-approximation to the ground-state. Collectively, our results reveal that the onset of the OGP is asymptotically tight: it matches the algorithmic threshold as $p\to\infty$. See Section~\ref{Results:Informal} for an overview of our results.


Our approach crucially draws on insights from the recent advances in the analysis of the $p$-spin glass model when $p$ is large, see~\cite{gamarnik2023shattering,alaoui2023sampling,alaoui2023shattering,alaoui2024near,anschuetz2024bounds} for a related body of work. 
We next make this connection explicit.
\subsection{Connections to the $p$-Spin Glass Model}\label{sec:Glass}
Given an $\boldsymbol{A}\in(\R^N)^{\otimes p}$ with i.i.d.\,\,$\cN(0,1)$ entries as in~\eqref{eq:A_Tensor} and a $\boldsymbol{\sigma}\in\{-1,1\}^N$, the Ising (pure) $p$-spin glass model is defined by the Hamiltonian, a degree-$p$ random polynomial:
\begin{equation}\label{eq:P-spin}
    H_N(\boldsymbol{\sigma}) = N^{-\frac{p-1}{2}}\sum_{1\le i_1,\dots,i_p\le N}A_{i_1,\dots,i_p}\boldsymbol{\sigma}(i_1)\cdots \boldsymbol{\sigma}(i_p).
\end{equation}
Here, the scaling $N^{-\frac{p-1}{2}}$ ensures non-trivial behavior. There is an extensive literature surrounding this model; we refer the interested reader to textbooks~\cite{talagrand2011mean} and~\cite{panchenko2013sherrington}, as well as a recent encyclopedic overview on the subject~\cite{charbonneau2023spin}. 

The $p$-spin model is notoriously challenging to analyze for fixed $p$. However, a key observation dating back to~\cite{derrida1980random,derrida1981random} asserts that its large $p$ limit is well approximated by the Random Energy Model (REM), where $H_N(\bs),\bs\in\{-1,1\}^N$ are centered i.i.d.\,normal random variables. This is based on a simple observation: for any $\bs\ne \pm \bs’$ and fixed $N$, the correlation between $H_N(\bs)$ and $H_N(\bs')$ defined in~\eqref{eq:P-spin} decays as $p\to\infty$. Since the REM consists of i.i.d.\,random variables, it is relatively straightforward to analyze. Analogously, rigorous calculations for the $p$-spin model become more tractable for large $p$, where the probability estimates can be approximated by independent random variables plus correction terms that vanish as $p\to\infty$. Strictly speaking, the REM corresponds to a different asymptotic regime where
$p\to\infty$ limit is taken before $N\to\infty$. Nevertheless, a recent body of work~\cite{gamarnik2023shattering,alaoui2023sampling,alaoui2023shattering,kizildaug2023sharp,alaoui2024near}, alongside an earlier work~\cite{talagrand2000rigorous} demonstrate that insights from the REM  remain powerful in the regime where the limit $N\to\infty$ is taken before $p\to\infty$. This latter regime is also our focus, where 
the idea of analyzing random tensors $\boldsymbol{A}$ with a large order parameter $p$ also enables us to study~\eqref{eq:MaxT}  when $k=\Theta(N)$. 

To illustrate the connection between~\eqref{eq:MaxT} and the $p$-spin glass model, consider the case where we seek to find a \emph{principal} subtensor with a large average entry (identical $I_1,\dots,I_p$). Suppose $I_1=\cdots=I_p:=I$ for some $I\subset[N]$, and define $\bs\in\{0,1\}^N$ as the indicator of $I$: $\bs(i)=1$ iff $i\in I$. Denoting by $\|\bs\|_0$ the number of non-zero entries of $\bs$,~\eqref{eq:MaxT} can be rewritten as
\begin{equation}\label{eq:BooleanSpinGlass}
\max_{\substack{\bs\in\{0,1\}^N \\ \|\bs\|_0=k}} H_{\mathrm{B}}(\bs)\quad\text{where}\quad H_{\mathrm{B}}(\bs)=\frac{1}{k^p} \sum_{1\le i_1,\dots,i_p\le N}A_{i_1,\dots,i_p}\bs(i_1)\cdots \bs(i_p).
\end{equation}
Modulo the sparsity constraint $\|\bs\|_0=k$ and the scaling $k^{-p}$,~\eqref{eq:BooleanSpinGlass} closely resembles the 
$p$-spin glass model~\eqref{eq:P-spin} with Boolean-valued $\bs$. In particular, for any $\bs\ne \bs'$, the correlation between $H_{\mathrm{B}}(\bs)$ and $H_{\mathrm{B}}(\bs')$ also decays exponentially as $p\to\infty$, suggesting 
that the large $p$ asymptotics may simplify
the analysis of~\eqref{eq:BooleanSpinGlass} as well. For the more general case where each $I_i$ need not be identical,~\eqref{eq:MaxT} is captured by what we term as the Boolean $p$-partite spin glass, see Section~\ref{sec:Future}. As our analysis will show, this model also becomes more tractable when $p$ is large.
\subsection{Submatrix Case: Prior Work}\label{sec:Submatrix-Results}
We now review existing results for the submatrix case. Throughout, $o(1)$ denotes a function tending to zero as $N,k\to\infty$. For $p=2$,~\cite{bhamidi2017energy} established that the ground-state of~\eqref{eq:MaxT} is of order $2(1+o(1))\sqrt{(\log N)/k}$ with high probability (w.h.p.) when $k=O(\log N/\log \log N)$. In fact, they obtain a more refined result, showing that the ground-state, appropriately centered and normalized, converges to a Gumbel distribution.\footnote{If $T$ is an exponential random variable with parameter $1$, $-\log T$ follows the Gumbel distribution.} The value, $2(1+o(1))\sqrt{(\log N)/k}$, corresponds to the maxima of $\binom{N}{k}^2$ i.i.d.\,random variables $\cN(0,k^{-2}$), see Remark~\ref{remark:E-max}.

\cite{bhamidi2017energy} also characterized the joint distribution of the $k\times k$ submatrix with the largest average and calculated the expectation and the variance of the number of locally optimal submatrices. Such submatrices arise as fixed points of a natural greedy algorithm called the large average submatrix ($\mathcal{L}\mathcal{A}\mathcal{S}$) algorithm proposed by~\cite{shabalin2009finding}. They conjectured that the $\mathcal{L}\mathcal{A}\mathcal{S}$ finds a submatrix whose average 
is a factor of $\sqrt{2}$ smaller than the ground-state, namely $(1+o(1))\sqrt{(2\log N)/k}$ w.h.p. This was later confirmed in~\cite{gamarnik2018finding}. In the same work, Gamarnik and Li also proposed an improved algorithm called the incremental greedy procedure ($\IGP$), showing that it finds, in polynomial time, a submatrix with an average value of $(1+o(1))(4/3)\sqrt{(2\log N)/k}$ w.h.p.\,when $k=o(\log^{1.5} N)$. In Section~\ref{sec:ALGS}, we review the $\IGP$ algorithm and extend it to subtensors.

Note that both $\mathcal{L}\mathcal{A}\mathcal{S}$ and $\IGP$ fall short of finding a globally optimal submatrix: they fail to identify a submatrix with an average value $2\sqrt{(\log N)/k}$, highlighting a \emph{statistical-computational gap}~\cite{bandeira2018notes,Gamarnik_2021,gamarnik2022disordered}. Using insights from spin glasses,~\cite{gamarnik2018finding} establish that the set of submatrices with an average $\alpha\sqrt{(2\log N) /k}$ is topologically connected when $\alpha<\alpha^* = 5\sqrt{2}/(3\sqrt{3})\approx 1.36$ but becomes disconnected when $\alpha>\alpha^*$. As such, the model exhibits the Overlap Gap Property (OGP) when $\alpha>\alpha^*$, leading to the conjecture that this problem is not solvable in polynomial time when $\alpha>\alpha^*$, which remains open. For additional results when $k=N^\gamma,\gamma\in(0,1)$, see also~\cite{cheairi2022densest}.

More recently,~\cite{Erba_2024} studied the large average submatrix problem when $k=\Theta(N)$, using non-rigorous replica method of the statistical mechanics. They uncovered a rich phase diagram, exhibiting  dynamical, static one-step replica symmetry breaking (RSB) and full-step RSB. For $k=\alpha N$ with $\alpha\to 0$, they show that the model exhibits the so-called Frozen 1-RSB phase. In particular, the Frozen 1-RSB phase coincides with efficient search algorithms (such as the $\IGP$). This behavior closely resembles binary perceptron, see~\cite{perkins2021frozen,abbe2022binary,abbe2022proof,gamarnik2022algorithms,gamarnik2023geometric} for detailed discussions. 

The large submatrix problem is a random optimization problem and 
falls within 
the class of unplanted models~\cite{wu2021statistical}. For planted versions or hypothesis testing variants, see~\cite{arias2011detection,balakrishnan2011statistical,kolar2011minimax,butucea-ingster,ma-wu,montanari2015limitation,chen2016statistical}.
\section{Overview of Our Results}\label{Results:Informal}
We now provide an overview of our results, established under the following assumption. \begin{Assumption}\label{Asm:Scaling} We assume $N\to\infty$ and that $k$ scales such that $\limsup_{N\to\infty}\frac{k}{N}<1$. \end{Assumption} Assumption~\ref{Asm:Scaling} is very mild and holds as long as $N-k =\Omega(N)$. Importantly, it covers the regime $k=\alpha N$ for all $\alpha\in(0,1)$, which is explicitly raised as an open question in~\cite{bhamidi2017energy}, as well as $k=o(N)$. The following quantity will be central to our work: \begin{equation}\label{eq:E-star} E_{\mathrm{max}}:=\sqrt{\frac{2p}{k^p}\log\binom{N}{k}}. \end{equation} {\textbf {Ground-State Value }} Our first result shows that the ground-state value $M^*$ in~\eqref{eq:MaxT} concentrates around $E_{\mathrm{max}}$ for large $p$. \begin{theorem}[Informal, see Theorem~\ref{thm:LargeTensor}]\label{thm:GS-Informal} For any $\epsilon>0$ and sufficiently large  $p$, $M^*\in \bigl[(1-\epsilon)E_{\mathrm{max}},E_{\mathrm{max}}\bigr]$ w.h.p. \end{theorem} Theorem~\ref{thm:GS-Informal} is valid both for $k=\Theta(N)$ and for $k=o(N)$.\footnote{In fact, for small $k$, it appears possible to prove Theorem~\ref{thm:GS-Informal} for all $p$---not just large $p$---via the second moment method. We omit the details for brevity.} The first step in our proof is to apply the second moment method, which alone only yields $M^*\ge (1-\epsilon)E_{\mathrm{max}}$ w.p.\,at least $\exp(-No_p(1))$, where $o_p(1)\to 0$ as $p\to\infty$. We then boost this to a high probability guarantee via an argument of~\cite{frieze1990independence}, which leverages the large $p$ asymptotics along with concentration properties of $M^*$ following from Borell-TIS inequality. See Section~\ref{sec:PfSketch} for an overview of proof. 

More broadly, this technique provides a powerful and general framework for proving sharp probabilistic bounds in models with at least two intrinsic parameters, particularly in regimes where one parameter is fixed and large while the other grows asymptotically. This structure appears widely in high-dimensional statistics, spin glasses, and theoretical computer science. However, this technique has seen only a handful of applications, primarily in the context of spin glasses~\cite{talagrand2011mean,anschuetz2024bounds,kizildaug2023sharp}, see also~\cite{kunisky2022strong} for an application to planted matchings. We believe this approach has the potential to serve as a powerful tool for models involving tensors. 
\begin{remark}\label{remark:E-max}
We provide some intuition for the value $E_{\mathrm{max}}$ in~\eqref{eq:E-star}. Recall the classical fact that if $X_1,\dots,X_M$ are i.i.d.\,\,$\mathcal{N}(0,\sigma^2)$, then $\max_{1\le i\le M}X_i=\sigma\sqrt{2\log M}(1+o(1))$ w.h.p. Consider $\binom{N}{k}^p$ random variables
    \[
    \frac{1}{k^p}\sum_{i_1\in I_1,\dots,i_p\in I_p}A_{i_1,\dots,i_p}\sim \mathcal{N}\left(0,\frac{1}{k^p}\right),
    \]
    where $I_1,\dots,I_p\subset[N]$ with $|I_i|=k$. While these random variables are not independent, our analysis shows that the underlying correlations vanish as $p\to\infty$. Consequently, $M^*$ behaves similarly to the maximum of $\binom{N}{k}^p$ i.i.d.\,\,$\mathcal{N}(0,\frac{1}{k^p})$ variables---which precisely corresponds to $E_{\mathrm{max}}$. As discussed in Section~\ref{sec:Glass}, this behavior closely resembles the Ising $p$-spin glass model whose large $p$ asymptotics is also well captured by the REM, where each corner of the discrete cube is endowed with an i.i.d.\,normal random variable.
\end{remark}
\paragraph{Algorithmic Guarantees} Our next focus is on algorithms. As mentioned, Gamarnik and Li~\cite{gamarnik2018finding} proposed a polynomial-time algorithm, $\IGP$, that identifies a $k\times k$ submatrix with an average value of $(4/3)\sqrt{2\log N/k}$ w.h.p., provided $k=o(\log^{1.5} N)$. We extend $\IGP$ to tensors, denoted by \IGPT\,(Algorithm~\ref{IGPT}) and establish the following.
\begin{theorem}[Informal, see Theorem~\ref{thm:IGPT}]\label{Thm:AlgInformal}
    Given $\boldsymbol{A}\in(\R^N)^{\otimes p}$ and $k=o(\log^{1.5} N)$, \IGPT\, finds a $k\times \cdots \times k$ subtensor with an average value of $(1+o(1))\frac{2\sqrt{p}}{p+1}E_{\mathrm{max}}$ w.h.p.
\end{theorem}
\begin{remark}\label{remark:Small-k}
Note that Theorem~\ref{Thm:AlgInformal} pertains to the small-$k$ regime, $k=o(\log^{1.5}N)$. This is natural in light of the current algorithmic landscape.
For the submatrix case, only two algorithms are known: (a) the $\IGP$~\cite{gamarnik2018finding}, which also requires $k=o(\log^{1.5}N)$, and (b) the $\mathcal{L}\mathcal{A}\mathcal{S}$~\cite{shabalin2009finding}, which has been rigorously analyzed only for constant $k$~\cite{bhamidi2017energy}. No provable guarantees are known beyond this regime, making it one of the open problems highlighted in Section~\ref{sec:Future}.
\end{remark}

The absence of any algorithmic improvement over $\IGP$ in the submatrix case suggests a statistical–computational gap between the ground-state value and the best known algorithmic guarantee. For subtensors, Theorem~\ref{Thm:AlgInformal} reveals that the \IGPT\, incurs a factor $\frac{2\sqrt{p}}{p+1}$ gap with respect to the ground-state. Motivated by a rich body of work on such algorithmic gaps (see the surveys~\cite{bandeira2018notes,Gamarnik_2021,gamarnik2022disordered}), we next delve deeper into this gap.

\paragraph{Overlap Gap Property \& Algorithmic Barriers} Our final focus is on the algorithmic barriers in solving~\eqref{eq:MaxT} approximately and efficiently. As is now customary for similar random optimization problems, our approach is based on the multi Overlap Gap Property ($m$-OGP). The $m$-OGP is an intricate geometrical property of the optimization landscape that has been used to rigorously rule out a wide array of algorithms for numerous random computational problems. 
At a high level, it asserts the absence of near-optimal solutions with certain specific intermediate overlap values, see Definition~\ref{Def:Overlap-Set-correlated-generic-subtensor} for a formal statement.  
Our main result to this end is as follows.
\begin{theorem}[Informal, see Theorem~\ref{thm:M-OGP-generic-tensor-correlated}]\label{thm:m-ogp-informal}
For any $\gamma>0$ and sufficiently large $p$, the large average subtensor problem exhibits $m$-OGP above the threshold $\gamma E_{\mathrm{max}}$.
\end{theorem}
The $m$-OGP is known to be a rigorous barrier for stable algorithms---a broad class that involves commonly used algorithms such as the approximate message passing~\cite{GamarnikJagannathAMP}, low-degree polynomials~\cite{gamarnik2020low,bresler2022algorithmic}, and low-depth Boolean circuits~\cite{gamarnik2024hardness}. Consequently, Theorem~\ref{thm:m-ogp-informal} directly implies the following:
\begin{theorem}[Informal]\label{thm:STABLE-FAIL}
    For any $\gamma>0$ and large enough $p$, sufficiently stable algorithms fail to find a subtensor with an average above $\gamma E_{\mathrm{max}}$.
\end{theorem}
The proof of Theorem~\ref{thm:STABLE-FAIL} follows  standard techniques in the literature on the OGP (see, e.g.,~\cite{GamarnikJagannathAMP,gamarnik2022algorithms,gamarnik2024hardness}) and is omitted here for brevity.
\paragraph{$m$-OGP for Large $p$ and Algorithmic Threshold} 
Theorem~\ref{Thm:AlgInformal} establishes that \IGPT\,identifies a subtensor with an average value $\frac{2\sqrt{p}}{p+1}E_{\mathrm{max}}$. Meanwhile, Theorems~\ref{thm:m-ogp-informal} and~\ref{thm:STABLE-FAIL} together imply the existence of a sequence $\{\gamma_p\}_{p\ge 1}$ with $\gamma_p\to 0$ as $p\to\infty$, such that the large average subtensor problem exhibits $m$-OGP above $\gamma_p E_{\mathrm{max}}$ for all sufficiently large $p$. Consequently, stable algorithms fail in this regime. 

This leads to an intriguing conclusion: the gap between the algorithmic threshold prescribed by the $m$-OGP and the performance of the \IGPT\,vanishes as $p\to\infty$, i.e., $|\gamma_p - 2\sqrt{p}/(p+1)|\to 0$. In other words, the $m$-OGP is asymptotically tight.
\begin{remark}\label{remark:Large-p}
 Theorems~\ref{thm:GS-Informal},~\ref{thm:m-ogp-informal} and~\ref{thm:STABLE-FAIL} hold in the large-$p$ regime. This is consistent with a broad line of work in random structures---such as random $k$-SAT, spin glasses, and random graphs---where rigorous results for fixed parameter values are quite challenging or remain open. In many of these models, substantial progress has first occurred in the large parameter regime. For instance, in the random $k$-SAT, both satisfiability and algorithmic thresholds are only known for large $k$~\cite{ding2015proof,bresler2022algorithmic}. In the Ising $p$-spin model, rigorous results for large $p$~\cite{talagrand2000rigorous} preceded breakthroughs for fixed $p$~\cite{talagrand2006parisi}, with the latter being far more delicate. The shattering phenomenon in spin glasses, a key structural property (see below), has only recently begun to be understood rigorously---again, primarily for large $p$~\cite{gamarnik2023shattering,alaoui2023shattering,alaoui2024near}.

Similarly, for (sparse) random graphs, most rigorous results—including Frieze’s original paper~\cite{frieze1990independence}, whose ideas we draw upon—pertain to the large average degree regime~\cite{bayati2010combinatorial,gamarnik2014limits,wein2020optimal}. While the large submatrix problem is well-motivated in applied statistics (e.g., biclustering), rigorous understanding remains limited. 
The subtensor case, to our knowledge, has not been studied rigorously at all. In light of these precedents, we view the large-$p$ regime as a natural and necessary starting point. Extension to fixed $p$ is among the open problems raised in Section~\ref{sec:Future}, which will likely require sophisticated tools from spin glass theory~\cite{talagrand2006parisi,talagrand2011mean,panchenko2013sherrington}.
\end{remark}
\paragraph{Comparison with~\cite{gamarnik2023shattering}} We compare our results with~\cite{gamarnik2023shattering}, establishing the $m$-OGP for the Ising $p$-spin glass and confirming the existence of a shattering phase---originally conjectured  in the landmark paper~\cite{kirkpatrick1987p}---using large-$p$ asymptotics. First,~\cite{gamarnik2023shattering} focuses on the Ising $p$-spin glass, whereas we study large subtensors, where the goal is to identify $p$-tuples of subsets $I_1,\dots,I_p$~\eqref{eq:MaxT}. Despite similarities in our approach to establishing $m$-OGP, our ground-state analysis differs significantly. \cite{gamarnik2023shattering} employs a standard second moment calculation, leveraging a Taylor expansion of the binary entropy. In contrast, applying the second moment in our setting is more intricate; it involves pairs of $p$-tuples of sets, requiring careful tracking of the size of intersections; see Proposition~\ref{Prop:2ndMom} and its proof in Section~\ref{pf:PropMain}.
More importantly, the second moment method alone only yields an $\exp(-No_p(1))$ probability bound, which we refine by using Frieze's argument alongside Borell-TIS inequality; see Section~\ref{sec:PfSketch} for the overview. Another key distinction is that~\cite{gamarnik2023shattering} primarily establishes existential results, whereas we also propose a polynomial-time algorithm for our model, Algorithm~\ref{IGPT}.
\paragraph{Concurrent Works} The same model has also been independently introduced in a concurrent work by Erba, Malo Kupferschmid, P\'{e}rez	Ortiz, and Zdeborov\'{a}~\cite{MAS-2025}, very recently posted on arXiv. Using non-rigorous tools from statistical physics, they characterize the equilibrium phase diagram, exhibiting replica symmetric and frozen one-step replica-symmetry-breaking phases for small and large values of the subtensor average, respectively. They also heuristically derive the algorithmic threshold for $\IGP$ adapted to tensors and carry out a non-rigorous OGP calculation. In addition to relying on non-rigorous methods, their work focuses on a different regime: fixed $p$, though under the crucial assumption $k\ll N$. 

A forthcoming work by Gamarnik and Gong~\cite{BranchingOGPSharp} establishes a matching algorithmic lower bound via a branching OGP argument. They show that the onset of branching OGP occurs precisely at $\frac{2\sqrt{p}}{p+1}$---the value achieved by \IGPT---suggesting that \IGPT\,\,is essentially optimal. Their argument builds on ideas from~\cite{du2025algorithmic}, see Remark~\ref{remark:branching-tight} for details.
\paragraph{Notation} For any set $A$, denote its cardinality by $|A|$. For any event $E$, $\ind\{E\}$ denotes its indicator. For any $v,v'\in\R^n$, $\ip{v}{v'}:=\sum_{1\le i\le n}v(i)v'(i)$. For $r>0$, $\log_r(\cdot)$ and $\exp_r(\cdot)$ respectively denote the logarithm and exponential functions base $r$; we omit the subscript when $r=e$. For $\boldsymbol{\mu}\in\R^k$ and $\Sigma\in\R^{k\times k}$, $\cN(\boldsymbol{\mu},\Sigma)$ denotes the multivariate normal distribution in $\R^k$ with mean $\boldsymbol{\mu}$ and covariance $\Sigma$. For any $q\in[0,1]$, $h(q):=-q\log_2 q-(1-q)\log_2(1-q)$ is the binary entropy. Given a matrix $M$, $\|M\|_F,\|M\|_2$, and $|M|$ respectively denote its Frobenius norm, spectral norm and determinant. Throughout, we employ the standard asymptotic notation, $\Theta(\cdot),o(\cdot),O(\cdot),\Omega(\cdot),\omega(\cdot)$. 
\section{Ground-State Value}\label{sec:ground-state} 
Recall  $M^*$ from~\eqref{eq:MaxT} and $E_{\mathrm{max}}$ from~\eqref{eq:E-star}. Under the Assumption~\ref{Asm:Scaling}, we establish the following.
\begin{theorem}\label{thm:LargeTensor}
  For any $\epsilon>0$, there exists a $P^*:=P^*(\epsilon)\in\mathbb{N}$ such that the following holds. Fix any $p\ge P^*$. Then,
  \[
  \mathbb{P}\Bigl[(1-\epsilon)E_{\mathrm{max}}\le M^* \le E_{\mathrm{max}}\Bigr]\ge 1-\Theta\left(\frac{1}{\sqrt{\log \binom{N}{k}}}\right) = 1-o_{N,k}(1).
  \]
\end{theorem}
Namely, in the large-$p$ regime, $M^*$ is approximately $E_{\mathrm{max}}$. The value $E_{\mathrm{max}}$ corresponds to the maximum of $\binom{N}{k}^p$ i.i.d.\,$\cN(0,k^{-p})$ random variables, as mentioned in Remark~\ref{remark:E-max}. When $k=o(N)$ grows sufficiently slowly in $N$, it appears possible to establish Theorem~\ref{thm:LargeTensor} for all $p$ (not just large $p$), through a direct application of the second moment method, which we skip.
\paragraph{Probability Estimate} When $k=o(N)$, the probability guarantee in Theorem~\ref{thm:LargeTensor} is of order $O(1/\sqrt{k\log N})$. When $k=\Theta(N)$, it is $O(1/\sqrt{N})$. A simple modification of our argument yields an improved probability guarantee: for any $\epsilon>0$ and large $p$, we have
\begin{equation}\label{eq:ImproveProb}
    \mathbb{P}\Bigl[(1-\epsilon)E_{\mathrm{max}}\le M^*\le (1+\epsilon)E_{\mathrm{max}}\Bigr]\ge 1-\binom{N}{k}^{-\Theta(1)},
\end{equation}
where the $\Theta(1)$ term depends only on $p$ and $\epsilon$. Notably, when $k=\Theta(N)$, the probability guarantee in~\eqref{eq:ImproveProb} is $1-\exp(-\Theta(N))$. 

We next provide a detailed overview of our proof; see Section~\ref{sec:PfLargeTensor} for the complete proof.
\subsection{Proof Overview for Theorem~\ref{thm:LargeTensor}}\label{sec:PfSketch}
Denote by $\F$ the set of all $p$-tuples $(I_1,\dots,I_p)$ of subsets of $[N]$, where $|I_i|=k$. Let $N_E$ count the number of $(I_1,\dots,I_p)\in \F$ with an objective value at least $E$. Observe from~\eqref{eq:MaxT} that $\{M^*\ge E\} = \{N_E\ge 1\}$. For the upper bound, $M^*\le E_{\mathrm{max}}$, we use the \emph{first moment method}, showing that $\mathbb{E}[N_{E_{\mathrm{max}}}] = o(1)$. Markov's inequality then yields
$$\mathbb{P}[N_{E_{\mathrm{max}}} = 0]\ge 1-\mathbb{E}[N_{E_{\mathrm{max}}}]=1-o(1).$$

The lower bound is more challenging and based on several different ideas outlined below.
\paragraph{Second Moment Method} The first ingredient for the lower bound is the \emph{second moment method}, based on the \emph{Paley-Zygmund Inequality}: for any non-negative integer-valued random variable $N$ with finite variance, $\mathbb{P}[N\ge 1]\ge \mathbb{E}[N]^2/\mathbb{E}[N^2]$.
In Proposition~\ref{Prop:2ndMom}, we apply this inequality to $N=N_E$ for $E:=(1-\epsilon)E_{\mathrm{max}}$. Estimating $\mathbb{E}[N_E^2]$ entails a summation over all pairs of $p$-tuples $\mathcal{I}=(I_1,\dots,I_p)\in\F$ and $\mathcal{I}'=(I_1',\dots,I_p')\in \F$ which is involved.  To this end, we highlight a reparameterization and the large-$p$ asymptotics. Note that for $\mathcal{I},\mathcal{I}'$ as above, random variables
\begin{equation}\label{eq:INFORMAL}
   \Phi_{\mathcal{I}}:= k^{-\frac{p}{2}}\sum_{i_1\in I_1,\dots,i_p\in I_p}A_{i_1,\dots,i_p} \quad\text{and}\quad \Phi_{\mathcal{I}'}:=k^{-\frac{p}{2}}\sum_{i_1\in I_1',\dots,i_p\in I_p'} A_{i_1,\dots,i_p}
\end{equation}
are both $\cN(0,1)$ with correlation $a_1\cdots a_p/k^p$, where $a_i := |I_i\cap I_i'|$. In general, the correlation can be as large as $1$. Suppose $\mathcal{I}$ and $\mathcal{I}'$ are such that sufficiently many $a_i$ are bounded away from $k$. More precisely, suppose that for some $\delta\in(0,1)$, the number of $a_i$ below $(1-\delta)k$ is equal to some function $\Theta_p(p)$. 
Then, $\mathbb{E}\bigl[\Phi_{\mathcal{I}}\cdot \Phi_{\mathcal{I}'}\bigr]\le (1-\delta)^{\Theta_p(p)}$, which vanishes as $p\to\infty$. In this case, the probability term can be estimated by two i.i.d.\,standard normal random variables with an additional correction term which we control using bounds from the literature~\cite{hashorva2003multivariate,hashorva2005asymptotics}. We couple this with a counting estimate (Lemma~\ref{lemma:BinomCoeff}) to show that $\mathbb{E}[N_E^2]$ is `dominated' by pairs $(\mathcal{I},\mathcal{I}')$ satisfying a correlation decay similar to above and that the contribution to $\mathbb{E}[N_E^2]$ from $(\mathcal{I},\mathcal{I}')$ having a substantial overlap is negligible. This yields,
\begin{equation}\label{eq:INFORMAL-2}    
\mathbb{P}[N_E\ge 1]\ge \frac{\mathbb{E}[N_E]^2}{\mathbb{E}[N_E^2]} \ge \exp\left(-c^p \bar{E}^2\right),\quad\text{where}\quad \bar{E} = k^{p/2}(1-\epsilon)E_{\mathrm{max}}
\end{equation}
for a $c<1$, ignoring various constants. Notably, this guarantee is exponentially small in $\log\binom{N}{k}$, indicating that the standard second moment method fails. However, as $p$ increases, the bound~\eqref{eq:INFORMAL-2} improves. Combined with a concentration property, this feature allows refining~\eqref{eq:INFORMAL-2}.
\paragraph{Repairing the Second Moment Bound via Borell-TIS inequality}  We now repair the bound~\eqref{eq:INFORMAL-2} by showing that $M^*$ is well-concentrated around its mean. The argument below goes back to~\cite{frieze1990independence}, who used it to determine the size of the largest independent sets in sparse \ER random graph $\mathbb{G}(n,\frac{d}{n})$ in the regime the average degree $d$ is large.

Note that $M^* = \max_{\mathcal{I}\in \F} k^{-\frac{p}{2}}\Phi_{\mathcal{I}}$. We first show, using the first moment method, that the process $\{\Phi_{\mathcal{I}}:\mathcal{I}\in \F\}$ is almost surely bounded. We then employ Borell-TIS inequality, a fundamental result in the study of Gaussian processes~\cite{adler2009random}, to obtain
\[\mathbb{P}\Bigl[M^*-\mathbb{E}[M^*]>u\Bigr]\le \exp\left(-\frac{u^2 k^p}{2}\right)\quad\text{and}\quad 
\mathbb{P}\Bigl[\Bigl|M^*-\mathbb{E}[M^*]\Bigr|>u\Bigr]\le 2\exp\left(-\frac{u^2 k^p}{2}\right)
\]
for any $u>0$. Take $u=\epsilon E_{\mathrm{max}}$. Since $c<1$, we obtain that for all sufficiently large $p$,
\begin{align*}
&\mathbb{P}\Bigl[M^*\ge (1-\epsilon)E_{\mathrm{max}}\Bigr] = \mathbb{P}\bigl[N_{(1-\epsilon)E_{\mathrm{max}}}\ge 1\bigr] \\
&\ge \exp\Bigl(-c^pk^p(1-\epsilon)^2 E_{\mathrm{max}}^2\Bigr) 
\ge \exp\left(-\frac{k^p\epsilon^2 E_{\mathrm{max}}^2}{2}\right) 
\ge \mathbb{P}\Bigl[M^*\ge \mathbb{E}[M^*]+\epsilon E_{\mathrm{max}}\Bigr],
\end{align*}
using~\eqref{eq:INFORMAL-2} and the Borell-TIS inequality. In particular,  $\mathbb{E}[M^*]\ge (1-2\epsilon)E_{\mathrm{max}}$. Applying now the second part of Borell-TIS inequality with the same $u$, we further conclude $M^*\ge \mathbb{E}[M^*]-u\ge (1-3\epsilon)E_{\mathrm{max}}$ w.h.p. Since $\epsilon$ is arbitrary, this yields Theorem~\ref{thm:LargeTensor} via a union bound. 
\section{Algorithmic Guarantees}\label{sec:ALGS}
We turn our attention to algorithms. For the submatrix case ($p=2$),~\cite{gamarnik2018finding} proposed the incremental greedy procedure ($\IGP$). For $k=o(\log^{1.5} N)$, the $\IGP$ finds in polynomial time a $k\times k$ submatrix with an average value of $(4/3)\sqrt{(2\log N)/k}$. Being the best known polynomial-time guarantee for the submatrix case, this result suggests a \emph{statistical-computational gap}: the ground-state value  is $2\sqrt{(\log N)/k}$, whereas the best polynomial-time algorithm finds a submatrix whose average is $(4/3)\sqrt{(2\log N)/k}$. This gap remains an open problem; improving algorithms or showing matching lower bounds are among the open questions raised in Section~\ref{sec:Future}.

In this section, we introduce the \IGPT, an extension of $\IGP$ to tensors, and establish that it identifies a subtensor with an average entry of $(1+o_{k,N}(1))\frac{2\sqrt{p}}{p+1}E_{\mathrm{max}}$. We begin by describing the $\IGP$, following~\cite{gamarnik2018finding}. 

\paragraph{$\IGP$: Informal Description} Given a matrix $A\in\R^{N\times N}$ with i.i.d.\,\,$\cN(0,1)$ entries, the $\IGP$ algorithm returns a $k\times k$ submatrix with a large average value through the following iterative procedure. Initialize by letting $i_1=1$ and selecting the column $j_1$ corresponding to the largest entry in the first row: $A_{1,j_1} = \max_{1\le j\le N}A_{1,j}$. Choose then the row $i_2\ne i_1$ that has the largest entry in column $j_1$: $A_{i_2,j_1} = \max_{i\ne i_1}A_{i,j_1}$. Next, identify a column $j_2\ne j_1$ maximizing the sum of the entries in rows $i_1$ and $i_2$: $A_{i_1,j_2} + A_{i_2,j_2} = \max_{j\ne j_1} A_{i_1,j} + A_{i_2,j}$.  Repeat this iteratively until $k$ rows $\{i_1,\dots,i_k\}$ and $k$ columns $\{j_1,\dots,j_k\}$ are selected. Gamarnik and Li introduce a slight modification to this procedure by partitioning $N$ indices into $k$ equally sized, disjoint blocks and restricting the choices of $i_t$ and $j_t$ to $t$. This adjustment circumvents dependencies between indices. Furthermore, 
as $k=o(\log^{1.5} N)$, $\log N=(1+o(1))\log(N/k)$, ensuring that the algorithmic guarantee remains unchanged. In Section~\ref{sec:AlgInformal}, we provide an informal argument based on integral approximation to justify that the $\IGP$ identifies a $k\times k$ submatrix with an average entry $(4/3)(1+o(1))\sqrt{2\log N/k}$.

\subsection{\IGPT: Extending $\IGP$ Algorithm to Tensors}\label{Sec:Extend-IGP}
Given $N,k\in\mathbb{N}$, we partition $[N]$ into $k+1$ disjoint subsets as follows. Let \begin{equation}\label{eq:Block-i}
P_i:=\bigl\{(i-1)\lfloor N/k\rfloor+1, (i-1)\lfloor N/k\rfloor+2,\dots,i\lfloor N/k\rfloor\bigr\},
\end{equation}
for $i\in[k]$. For $i=k+1$, we set $P_{k+1}=\varnothing$ if $N$ is divisible by $k$; otherwise $P_{k+1}:=\{k\lfloor N/k\rfloor+1,\dots,N\}$, which will not be used. Our algorithm is as follows.
\begin{algorithm}
\caption{Incremental Greedy Procedure for Tensors (\IGPT)}\label{IGPT}  
\begin{algorithmic}
\STATE \textbf{Input}: An order-$p$ tensor $\boldsymbol{A}\in(\R^N)^{\otimes p}$ and an integer $k\ge 1$.
\STATE \textbf{Initialization (Step 1)}: For $1\le u\le p-1$, select $i^{(u)} \in P_1$ arbitrarily, set $I_u = \{i^{(u)}\}$. 
\STATE \textbf{Initialization (Step 2)}: Let $i^{(p)} = \arg\max_{j\in P_1} A_{i^{(1)},\dots,i^{(p-1)},j}$. Set $I_p = \{i^{(p)}\}$.
\STATE \textbf{Loop-1}: Repeat until $|I_1|=|I_2|=\cdots=|I_p|=k$.  \\
$\quad$ \textbf{Loop-2}: For each $1\le u\le p$, let
\[
i^{(u)}  = \arg\max_{j\in P_{|I_u|+1}} \sum_{\substack{i_1\in I_1,\dots,i_{j-1}\in I_{j-1} \\ i_{j+1}\in I_{j+1},\dots,i_p\in I_p}} A_{i_1,\dots,i_{j-1},j,i_{j+1},\dots,i_p},
\]
$\quad$ breaking ties arbitrarily.  Set $I_u=I_u\cup\{i^{(u)}\}$.
\STATE \textbf{Output}: $I_1,\dots,I_p$.
\end{algorithmic}
\end{algorithm}

Observe that when $p$ is constant, the runtime of \IGPT\, is polynomial in $N$ and $k$. Our next result establishes the following guarantee for the output of the \IGPT. 
\begin{theorem}\label{thm:IGPT}
Suppose that $\boldsymbol{A}$ has i.i.d.\,$\cN(0,1)$ entries, $k=o(\log^{1.5} N)$, and  $(I_1,\dots,I_p)$ is the output of \IGPT\,\,on $\boldsymbol{A}$. There exists an $N_0$ such that for all $N\ge N_0$, 
\[
\frac{1}{k^p} \sum_{i_1\in I_1,\dots,i_p\in I_p}A_{i_1,\dots,i_p} = \left(1+o_{k,N}(1)\right)\frac{2\sqrt{p}}{p+1}E_{\mathrm{max}}
\]
with probability $1-O(k/\log^{1.5}N)$.
\end{theorem}
See Section~\ref{IGPT-Proof} for the proof. Theorem~\ref{thm:IGPT} asserts that \IGPT\,\,finds a subtensor with an average value $O_p(\frac{1}{\sqrt{p}})E_{\mathrm{max}}$ w.h.p., provided $k=o(\log^{1.5}N)$. The $o_{k,N}(1)$ term converges to zero as $k,N\to\infty$; it can be inferred from our proof. 

Our argument is based on (a) a quantitative version of the fact that the maxima of i.i.d.\,standard normals, appropriately centered and normalized, converges to a Gumbel distribution, and (b) an integral approximation. We reproduce the former fact from~\cite{gamarnik2018finding} as Lemma~\ref{Lemma:Gauss-Max}. Importantly, Lemma~\ref{Lemma:Gauss-Max} holds with probability $1-O(\log^{-1.5}N)$, leading to the requirement $k=o(\log^{1.5} N)$.\footnote{The same technical assumption was also required in~\cite{gamarnik2018finding}.} Extension beyond this regime is left for future work.
\section{Algorithmic Barriers from the Overlap Gap Property}\label{sec:OGP}
We identified the ground-state value $E_{\mathrm{max}}$ in Section~\ref{sec:ground-state} and introduced the \IGPT\,\,in Section~\ref{sec:ALGS}. Theorem~\ref{thm:IGPT} establishes that \IGPT\, identifies a subtensor with an average entry of $\frac{2\sqrt{p}}{p+1}E_{\mathrm{max}}$, which falls short of optimality by an  $O_p(\frac{1}{\sqrt{p}})$ factor. Recall that \IGPT\,\,is an extension of $\IGP$~\cite{gamarnik2018finding}, originally developed for submatrices,  where a similar gap between the ground-state and algorithmic guarantees also exists. These observations suggest a \emph{statistical-computational gap} for the large average subtensor problem, raising several fundamental questions. How well can algorithms perform? What are the inherent algorithmic barriers?

Our next focus is on the algorithmic lower bounds for this model. A widely used approach for proving algorithmic hardness of random optimization problems similar to~\eqref{eq:MaxT} is to show that near-optimal solutions exhibit the Overlap Gap Property (OGP)---a fundamental barrier for a broad class of algorithms. Introduced in~\cite{gamarnik2014limits}, the OGP approach has since led to nearly sharp algorithmic lower bounds for a broad array of random computational problems, see, e.g.,~\cite{gamarnik2020low,bresler2022algorithmic,wein2020optimal,gamarnik2022algorithms,gamarnik2023geometric,gamarnik2023shattering,gamarnik2023algorithmic,li2024discrepancy,gamarnik2025optimal,du2025algorithmic,huang2025tight,huang2025strong,huang2025statistical,sellke2025tight}. The literature on the OGP is extensive; we do not attempt a comprehensive overview. Instead, we refer the reader to the survey~\cite{Gamarnik_2021} and to introductions of~\cite{bresler2022algorithmic,gamarnik2022algorithms,gamarnik2023shattering}.

In what follows, we show that near ground-states of our model exhibit the ensemble multi OGP ($m$-OGP). We begin by formally defining the set of near ground-states under investigation.
\begin{definition}\label{Def:Overlap-Set-correlated-generic-subtensor}
    Given $m\in\mathbb{N}$,  $\frac{1}{2}<\nu_1<\nu_2<1$, $\gamma>0$, and $\mathcal{J} \subset [0,\pi/2]$, define by $S_\gamma(m,\nu_1,\nu_2,\mathcal{J})$ the set of all $m$-tuples $\left(\mathcal{I}_1,\dots,\mathcal{I}_m\right)$ satisfying the following conditions:
    \begin{itemize}
        \item [(a)] For any $i\in[m]$, the multi-index $\mathcal{I}_{i}$ consists of $p$-tuples $\mathcal{I}_{i} = \left(I_{1,i},\ldots, I_{p,i}\right)$ such that for any $1\le j\le p$, $I_{j,i}\subset[N]$ and $|I_{j,i}|=k$.
        \item[(b)] For any $1\le i_1<i_2\le m$, 
        \[
        {\frac{|I_{j,i_1}\cap I_{j,i_2}|}{k} \in[\nu_1,\nu_2],
        \quad \forall\, 1\leq j \leq p}.
        \]
        \item[(c)] There exists $\tau_{1},\ldots,\tau_{m} \in \mathcal{J}$ such that 
        \[
        \frac{1}{k^{p/2}} \sum_{i_1\in I_{1,\ell},\dots,i_p\in I_{p,\ell}} \hat{A}^{(\ell)}_{i_1,\dots,i_p} (\tau_\ell)\ge \gamma k^{p/2} E_{\mathrm{max}}
        \,,
        \quad 
        \forall 1 \leq \ell \leq m\,,
        \]
        where for any $1\leq \ell \leq m$, $\tau \in [0,\pi/2]$ and $1\le i_1,\dots,i_p\le N$ 
        \begin{align}\label{eq:Interpolate}
            \hat{A}^{(\ell)}_{i_1,\dots,i_p}(\tau)
            &=
            \cos(\tau)
            A^{(0)}_{i_1,\dots,i_p}
            +
            \sin(\tau)
            A^{(\ell)}_{i_1,\dots,i_p}
        \end{align}
        for i.i.d $A^{(\ell)}_{i_1,\ldots,i_p}\sim \mathcal{N}(0,1)$.
    \end{itemize}
\end{definition}
The set $S_\gamma(m,\nu_1,\nu_2,\mathcal{J})$ consists of $m$-tuples of near ground-states w.r.t.\,correlated tensors $\hat{A}^{(\ell)}(\tau)$. Each solution $\mathcal{I}_i$ is a $p$-tuple of size-$k$ subsets of $[N]$, representing indices of a $k\times \cdots \times k$, order-$p$ subtensor. The parameter $\gamma$ quantifies the proximity to the ground-state value $E_{\mathrm{max}}$; specifically, $S_\gamma$ consists of solutions exceeding the threshold $\gamma E_{\mathrm{max}}$. The set $\mathcal{J}$ is used for generating correlated instances through the interpolation scheme given in~\eqref{eq:Interpolate}. 

The parameters $\nu_1<\nu_2$ impose a constraint on the intersections of the index sets. Specifically, for any $\mathcal{I}_{i_1},\mathcal{I}_{i_2}$ and any coordinate $1\le j\le p$, condition {(b)} asserts that the corresponding index sets, $\mathcal{I}_{j,i_1}$ and $\mathcal{I}_{j,i_2}$ have an `intermediate' overlap. Thus,  $S_\gamma$ consists of near-optimal $m$-tuples with intermediate pairwise  overlaps. Whenever $S_{\gamma}$ is empty, there is no $m$-tuple of near-optima with intermediate overlaps, highlighting an \emph{overlap gap}.

Our final main result establishes that the large average subtensor problem exhibits the ensemble $m$-OGP when $p$ is large: $S_\gamma$ is empty with suitable $m,\nu_1$ and $\nu_2$. 
\begin{theorem}\label{thm:M-OGP-generic-tensor-correlated}
For any $m\in\mathbb{N}$ and $\gamma>1/\sqrt{m}$, there exists $P:=P(\gamma)\in\mathbb{N}$, $1/2<\nu_1<\nu_2<1$ and $c>0$ such that the following holds. Fix any $p\ge P$ and $\mathcal{J} \subset [0,\pi/2]$ with $|\mathcal{J}|\le \binom{N}{k}^c$. 
    Then,
    \[
\mathbb{P}\Bigl[S_\gamma(m,\nu_1,\nu_2,\mathcal{J})=\varnothing\Bigr]\ge 1-\binom{N}{k}^{-\Theta(1)}
    \]
    for all sufficiently large $k$ and $N$, where $\Theta(1)$ term depends only on $p,\gamma$ and $m$.
\end{theorem}
\paragraph{Proof Sketch} Our proof relies on the \textit{first moment method}. 
Informally, we show that the tails of correlated $m$-tuples of near-optimal solutions behave as $m$ i.i.d.\,centered Gaussians with an additional correction term that vanishes as $p\to\infty$. We establish this using multivariate normal tail bounds~\cite{hashorva2003multivariate,hashorva2005asymptotics} as well as Slepian's Gaussian comparison inequality~\cite{slepian1962one}. We combine the resulting tail bound with a counting estimate on the number of admissible $m$-tuples that satisfy the overlap constraints, i.e., conditions {$\mathrm {(a)}$}-{$\mathrm (b)$} in Definition~\ref{Def:Overlap-Set-correlated-generic-subtensor} to show that $\mathbb{E}\left[\left|S_\gamma(m,\nu_1,\nu_2,\mathcal{J})\right|\right] = \binom{N}{k}^{-\Theta(1)}$ for suitable parameters. Markov's inequality then yields Theorem~\ref{thm:M-OGP-generic-tensor-correlated}. For the complete proof, see Section~\ref{sec:proof-MOGP}.

\paragraph{Algorithmic Lower Bounds and Sharpness of $m$-OGP} 
The ensemble $m$-OGP established in Theorem~\ref{thm:M-OGP-generic-tensor-correlated} is a rigorous barrier for a broad class of algorithms that exhibit input stability.\footnote{Informally, an algorithm $\mathcal{A}$ is stable if for $\boldsymbol{A}$ with a small $\|\boldsymbol{A}-\boldsymbol{A}'\|$, the outputs $\mathcal{A}(\boldsymbol{A})$ and $\mathcal{A}(\boldsymbol{A}')$ are close in an appropriate metric. For a formal statement, see, e.g.,~\cite{gamarnik2022algorithms,gamarnik2023algorithmic}.} This class includes widely used algorithms such as approximate message passing~\cite{GamarnikJagannathAMP}, low-degree polynomials~\cite{gamarnik2020low,bresler2022algorithmic}, and low-depth Boolean circuits~\cite{gamarnik2024hardness}, among others. In our setting, Theorem~\ref{thm:M-OGP-generic-tensor-correlated} implies that for any $\gamma>0$, there is a $P_\gamma\in\mathbb{N}$ such that for any $p\ge P_\gamma$, sufficiently stable algorithms fail to find a subtensor with an average entry above $\gamma E_{\mathrm{max}}$. This conclusion follows by directly adapting the techniques of~\cite{GamarnikJagannathAMP,gamarnik2022algorithms,gamarnik2024hardness}, which we skip for brevity. 

Next, we demonstrate that the $m$-OGP is asymptotically tight. From Theorem~\ref{thm:IGPT}, \IGPT\, identifies subtensors with an average entry $O_p(\frac{1}{\sqrt{p}})E_{\mathrm{max}}$. Simultaneously, Theorem~\ref{thm:M-OGP-generic-tensor-correlated} establishes the existence of a sequence $(\gamma_p)_{p\in\mathbb{N}}$ with $\lim_{p\to\infty}\gamma_p = 0$, such for tensors of order $p$, the model exhibits $m$-OGP above the threshold $\gamma_p E_{\mathrm{max}}$. Taken together, we conclude that $m$-OGP is asymptotically sharp: the gap between the $m$-OGP threshold and the algorithmic guarantee vanishes as $p\to\infty$. This behavior mirrors the Ising $p$-spin glass model where the algorithmic threshold for its formal $p\to\infty$ limit---namely the REM---is at zero~\cite{addario2020algorithmic}, and the $m$-OGP threshold similarly approaches to zero as $p\to\infty$~\cite{gamarnik2023shattering}.
\begin{remark}\label{remark:branching-tight}
We highlight a forthcoming work by David Gamarnik and Shuyang Gong~\cite{BranchingOGPSharp}. They show that the onset of the branching OGP---a sophisticated variant of the OGP---occurs precisely at $\frac{2\sqrt{p}}{p+1}$, yielding a matching lower bound and suggesting that \IGPT\,\,is essentially optimal. The branching OGP was originally introduced in the context of spin glasses~\cite{huang2025tight}, where it led to the sharpest known algorithmic lower bounds. The argument of~\cite{BranchingOGPSharp} leverage the online nature of \IGPT---specifically, its sequential selection of the largest entries---along with a branching OGP framework tailored online algorithms for the graph alignment problem~\cite{du2025algorithmic}.
\end{remark}
\section{Open Questions and Future Work}\label{sec:Future}
Our work introduces the large average subtensor problem and identifies several promising research directions.
Extending beyond Gaussian ensembles---e.g., to tensors with i.i.d.\,entries drawn from sufficiently regular distributions---remains an open avenue. The argument outlined in Section~\ref{sec:PfSketch} would still apply, provided a second moment guarantee as in~\eqref{eq:INFORMAL-2} and concentration of $M^*$ hold. In the Gaussian case, the latter follows from Borell-TIS inequality; for sub-Gaussians, techniques such as chaining might help, see~\cite{talagrand2005generic,ledoux2013probability,van2014probability}.

\paragraph{Submatrix Case ($p=2$) and Fixed $p$} As noted, rigorous results for the submatrix case when $k=\Theta(N)$ remain elusive. 
A potential approach is to map Boolean spins $\bs$ to Ising ones using the transformation $2\bs -\boldsymbol{1}\in\{-1,1\}^N$, where $\boldsymbol{1}$ is the all-ones vector. This transformation reformulates the problem as a Sherrington-Kirkpatrick model with an external field, which can be analyzed using advanced spin glass techniques, such as those in~\cite{talagrand2006parisi,talagrand2011mean,panchenko2013sherrington}. We leave this ambitious yet challenging question for future work. Our results hold for large $p$; extending them to fixed $p$ presents another intriguing challenge where mapping to Ising $p$-spin model with an external field again offers a promising starting point.
\paragraph{Range of Algorithmic Guarantee} 
The \IGPT\, and its submatrix counterpart, $\IGP$, require $k=o(\log^{1.5}N)$. Similarly, the earlier algorithm $\mathcal{L}\mathcal{A}\mathcal{S}$ has been rigorously analyzed only for fixed $k$, as noted in Remark~\ref{remark:Small-k}. These remain the only known algorithms for both the submatrix and subtensor cases. Extending their applicability beyond $k = o(\log^{1.5} N)$ is an intriguing direction for future work.
\paragraph{Boolean Spin Glasses} As discussed in Section~\ref{sec:Glass}, the problem of finding a principal subtensor with a large average can be reformulated as solving~\eqref{eq:BooleanSpinGlass}. 
In analogy with the classical $p$-spin glass model, we term this the \emph{Boolean $p$-spin glass} model. For the more general case where $I_i$ need not be identical, denote by $\bs_i\in\{0,1\}^N$ the indicator of $I_i$. In this setting,~\eqref{eq:MaxT} can be rewritten as:
\[
\max_{\substack{\bs_i\in\{0,1\}^N \\ \|\bs_i\|_0 = k \\ 1\le i\le p}}\frac{1}{k^p}\sum_{1\le i_1,\dots,i_p\le N}A_{i_1,\dots,i_p}\bs_1(i_1)\cdots \bs_p(i_p),
\]
which we term the \emph{Boolean $p$-partite spin glass} model. (The case $p=2$ is known as bipartite spin glass, see~\cite{auffinger2014free,mckenna2024complexity} for the \emph{spherical case}, $\|\bs\|_2=\sqrt{N}$.) Both the Boolean $p$-spin glass and its $p$-partite counterpart are intriguing models in their own rights. 

To the best of our knowledge, these model---except for the Boolean SK model~\cite{albanese2024boolean}---have not been considered in the physics literature. Both models (with or without the sparsity constraints) offer exciting directions for further study, particularly for physicists and probabilists.

\section{Performance of $\IGP$: An Informal Argument}\label{sec:AlgInformal}
Recall that the $\IGP$ identifies a submatrix with an average value of $(1+o(1))(4/3)\sqrt{2\log N/k}$ w.h.p. Below, we present an informal calculation to justify this guarantee, reproducing material  from~\cite{gamarnik2018finding}.  In what follows, we omit dependencies between random variables, which can be circumvented (see below). 

We estimate $\frac{1}{k^2}\sum_{i\in I,j\in J}A_{i,j}$, using standard results on the maxima of independent normal random variables (see, e.g., Remark~\ref{remark:E-max}). Note that $A_{i_1,j_1}$ is of order $\sqrt{2\log n}$. Similarly, $A_{i_2,j_1}$ is of order $\sqrt{2\log (n-1)}$. Continuing, suppose that $t$ row indices $i_1,\dots,i_t$ and $t-1$ column indices $j_1,\dots,j_{t-1}$ are found, and the column index $j_t$ is defined as $A_{i_1,j_t}+\cdots+A_{i_t,j_t} = \max_{j\notin\{j_1,\dots,j_{t-1}\}} A_{i_1,j}+\cdots+A_{i_t,j}$.
This corresponds to finding the maxima of $N-(t-1)$ $\cN(0,t)$ random variables, which is of order $\sqrt{t}\sqrt{2\log (N-(t-1)} = (1+o(1))\sqrt{t}\sqrt{2\log N}$, using $t\le k = o(\sqrt{\log N})$. Iterating this reasoning together with the integral approximation, $\sum_{1\le k\le n}\sqrt{k} =\int_1^k \sqrt{t}\;dt +O(k)$, we obtain 
\[
\frac{1}{k^2}\sum_{i\in I,j\in J}A_{i,j} =\frac{2(1+o(1))}{k^2}\sqrt{2\log N}\sum_{1\le t\le k}\sqrt{t} = \frac{4}{3}(1+o(1))\sqrt{\frac{2\log N}{k}}.
\]
This reasoning ignores dependencies. Nevertheless,~\cite{gamarnik2018finding} circumvent this issue by partitioning $N$ indices into $k$ equally sized blocks and restricting their search to $t^{\mathrm{th}}$ block at step $t$, as described in Section~\ref{sec:ALGS}. The algorithm \IGPT\, follows the same approach.
\section{Proofs}
\subsection{Proof of Theorem~\ref{thm:LargeTensor}}\label{sec:PfLargeTensor}
In this section, we prove Theorem~\ref{thm:LargeTensor}. We begin by introducing some notation. Define 
\begin{equation}\label{eq:F-of-k}
    \F:=\Bigl\{\mathcal{I}=(I_1,\dots,I_p):I_i\subset[N],|I_i|=k,1\le i\le p\Bigr\}.
\end{equation}
Here, $\mathcal{I}=(I_1,\dots,I_p)$ is a multi-index notation. With this, we set
\begin{equation}\label{eq:N__E}
    N_E :=\sum_{\mathcal{I}\in \F} \ind\Bigl\{\Phi_{\mathcal{I}}>Ek^{p/2}\Bigr\},\quad\text{where}\quad  \Phi_{\mathcal{I}}=\frac{1}{k^{p/2}}\sum_{i_1\in I_1,\dots,i_p\in I_p}A_{i_1,\dots,i_p}.
\end{equation}
Observe that $\Phi_{\mathcal{I}}\sim \cN(0,1)$ for all $\mathcal{I}\in \F$, $M^* = \max_{\mathcal{I}\in \F}\Phi_{\mathcal{I}}/k^{p/2}$ and that for all $E$,
\[
\{M^* \ge E \} = \{N_E\ge 1\}.
\]
We next recall the Gaussian tail bound: for any $x>0$,
\begin{equation}\label{eq:GaussTail}
     \frac{x}{x^2+1}\frac{\exp(-x^2/2)}{\sqrt{2\pi}} \le \mathbb{P}[\cN(0,1)\ge x]\le \frac{\exp(-x^2/2)}{x\sqrt{2\pi}}.
    \end{equation}
Notice in particular that for $x=\omega(1)$, 
\[
    \mathbb{P}[\cN(0,1)\ge x] = \frac{\exp(-x^2/2)}{x\sqrt{2\pi}} \bigl(1+o(1)\bigr).
\]
\paragraph{First Moment Estimate} Observe that
\begin{align*}          
\mathbb{E}[N_E]
&=
\sum_{\mathcal{I}\in \F} \mathbb{P}\bigl[\Phi_{\mathcal{I}}>Ek^{p/2}\bigr]
=
\binom{N}{k}^p \mathbb{P}\Bigl[ \mathcal{N}(0,1) > k^{p/2} E\Bigr].
\end{align*}
Using~\eqref{eq:GaussTail}, we obtain
\[
\mathbb{E}[N_{E_{\mathrm{max}}}] = \Theta\left(\frac{1}{\sqrt{p\log \binom{N}{k}}}\right).
\]
As $\binom{N}{k} = \omega(1)$ per Assumption~\ref{Asm:Scaling}, Markov's inequality yields
\begin{equation}\label{eq:LargeT-Upper}
   \mathbb{P}\bigl[M^*\ge E_{\mathrm{max}}\bigr] = \mathbb{P}[N_{E_{\mathrm{max}}}\ge 1] \le \mathbb{E}[N_{E_{\mathrm{max}}}] = \Theta\left(\frac{1}{\sqrt{p\log \binom{N}{k}}}\right) = o(1). 
\end{equation}
This establishes the upper bound in Theorem~\ref{thm:LargeTensor}. 

We turn our attention to the lower bound. To that end, we establish a technical result.
\begin{proposition}\label{Prop:2ndMom}
   Suppose that Assumption~\ref{Asm:Scaling} holds and $\alpha>0$ is any number satisfying $\frac{N}{k}\ge 1+\alpha$. Fix an arbitrary $\epsilon>0$. Let $f(p)=\epsilon p/2$, $E=(1-\epsilon)E_{\mathrm{max}}$ and
    \begin{equation}\label{eq:Bar-of-E}
          \bar{E} = (1-\epsilon)k^{p/2}E_{\mathrm{max}} = (1-\epsilon)\sqrt{2p\log \binom{N}{k}}.
      \end{equation}
   There exists a $\delta\in(0,1)$, depending only on $\alpha$ and $\epsilon$ such that 
   \[
   \mathbb{P}\left[N_E\ge 1\right]\ge 
    \left(\frac{(1+(1-\delta)^{f(p)})^2}{\sqrt{1-(1-\delta)^{2f(p)}}}\exp\Bigl(\left(1-\delta\right)^{f(p)}\bar{E}^2\Bigr)
    +
    \binom{N}{k}^{-\Theta(1)}\right)^{-1},
   \]
   where the $\Theta(1)$ term depends only on $p$ and $\epsilon$.
\end{proposition}
The proof of Proposition~\ref{Prop:2ndMom} is delicate and based on the \emph{second moment method}, i.e., the Paley-Zygmund inequality, as well as a tracking a certain sum running over $p$-tuples $\mathcal{I}=(I_1,\dots,I_p)\in\F$ of sets $|I_i|=k,1\le i\le N$. See Section~\ref{pf:PropMain} for details.

We now assume Proposition~\ref{Prop:2ndMom} is valid, and prove  Theorem~\ref{thm:LargeTensor}. To that end, we record the following concentration result.
\begin{theorem}[Borell-TIS Inequality]\label{thm:borell-TIS}
    Let $T$ be a topological space and $(f_t)_{t\in T}$ be a centered Gaussian process on $T$ such that $\|f\|_T:=\sup_{t\in T}|f_t|$ is bounded almost surely. Let $\sigma_T^2 = \sup_{t\in T}\mathbb{E}[f_t^2]$. Then, $\mathbb{E}[\|f\|_T]$ and $\sigma_T^2$ are both finite, and for all $u>0$,
    \[
    \mathbb{P}\left[\sup_{t\in T}f_t>\mathbb{E}\left[\sup_{t\in T}f_t\right]+u\right]\le \exp\left(-\frac{u^2}{2\sigma_T^2}\right),
    \]
    and from the symmetry
    \[
    \mathbb{P}\left[\left|\sup_{t\in T}f_t-\mathbb{E}\left[\sup_{t\in T}f_t\right]\right|>u\right]\le 2\exp\left(-\frac{u^2}{2\sigma_T^2}\right).
    \]
\end{theorem}
Theorem~\ref{thm:borell-TIS} is a crucial tool in the study of Gaussian processes. For details, see~\cite{adler2009random}; see also the original papers~\cite{borell1975brunn} and~\cite{ibragimov1976norms}. We apply Theorem~\ref{thm:borell-TIS} to the collection $k^{-p/2}\Phi_{\mathcal{I}}$, $\mathcal{I}\in \F$. To that end, we verify that
\[
\mathbb{E}\left[\left(k^{-p/2}\Phi_{\mathcal{I}}\right)^2\right] = k^{-p},\quad\forall \mathcal{I}\in \F.
\]
Moreover, as $\max_{i_1,\dots,i_p}|A_{i_1,\dots,i_p}| = O(\sqrt{2p \log N})$ w.h.p.\,(which can be established by using the first moment method), we have that $\sup_{\mathcal{I}\in \F}|k^{-p/2}\Phi_{\mathcal{I}}| = O(\sqrt{2p \log N})$. Consequently, the process is almost surely bounded for fixed $p,N$.  As
\[
M^* = \max_{\mathcal{I}\in \F} \frac{\Phi_{\mathcal{I}}}{k^{p/2}},
\]
$M^*$ is concentrated around its mean due to Theorem~\ref{thm:borell-TIS}.\paragraph{Repairing Second Moment Bound via Concentration}
We now repair the bound $\mathbb{P}[N_E\ge 1]\ge \exp\bigl(-\bar{E}^2 o_p(1)\bigr)$ arising in Proposition~\ref{Prop:2ndMom} by using the fact that $M^*$ is well-concentrated around its mean. This argument is originally due to~\cite{frieze1990independence}, who used it in the context of optimization over random graphs.

For $f(p)=\epsilon p/2$, set
\begin{align*}
    u^* := \sqrt{2+2\delta}\cdot (1-\delta)^{\frac{f(p)}{2}}(1-\epsilon)\sqrt{\frac{2p}{k^p}\log\binom{N}{k}}\,.
\end{align*}
Notice that there exists a $P_2$ such that for all $p\ge P_2$ and $k,N$ sufficiently large, we have:
\begin{align}
    \mathbb{P}[M^*\ge E]&=\mathbb{P}[N_E\ge 1] \nonumber\\
    &\ge \left(\frac{(1+(1-\delta)^{f(p)})^2}{\sqrt{1-(1-\delta)^{2f(p)}}}\exp\Bigl(\left(1-\delta\right)^{f(p)}\bar{E}^2\Bigr)
    +
    \binom{N}{k}^{-\Theta(1)}\right)^{-1} \label{eq:USE-2nd-MOM-BD}\\
    &\ge  \exp\Bigl(-(1+\delta)(1-\delta)^{f(p)}\bar{E}^2\Bigr) \label{eq:LARGE-P} \\
    &= \exp\left(-\frac{(u^*)^2k^p}{2}\right) \nonumber\\
    &\ge \mathbb{P}\bigl[M^*\ge \mathbb{E}[M^*]+u^*\bigr]\label{BORELL-TIS},
\end{align}
where~\eqref{eq:USE-2nd-MOM-BD} follows from Proposition~\ref{Prop:2ndMom},~\eqref{eq:LARGE-P} is valid for all $p\ge P_2$ and $k,N$ large due to the fact that $(1-\delta)^{f(p)} = (1-\delta)^{p\epsilon/2}\to 0$ as $p\to\infty$, and~\eqref{BORELL-TIS} uses the first part of Theorem~\ref{thm:borell-TIS} with $u=u^*$. In particular, 
\begin{align}\label{eq:E-of-M-star-lower}
    \mathbb{E}[M^*] \geq E- u^*
    =
   E_{\mathrm{max}} \cdot (1-\epsilon)\left( 1- \sqrt{2+2\delta}\cdot (1-\delta)^{f(p)/2} \right)\,.
\end{align}
We again note that there exists a $P_{3}\in\mathbb{N}$ such that for all $p\geq P_{3}$
\begin{align}\label{eq:3eps/2}
    (1-\epsilon)\left(1-\sqrt{2+2\delta}\cdot \left(1-\delta\right)^{\frac{f(p)}{2}}\right)\ge \left(1-\frac{3\epsilon}{2}\right)\,.
\end{align}
We can now complete the proof.
Note that for all $p\ge \max\{P_2,P_3\}$,
\begin{align}
    \mathbb{P} \left[M^{*} 
    \geq \left(1 - 2 \epsilon \right) E_{\mathrm{max}} \right]
    &\geq 
      \mathbb{P} \left[ M^{*} \geq \mathbb{E}[M^*] - \frac{\epsilon}{2} E_{\mathrm{max}}  \right] \label{eq:E111} \\
    &\ge 
    1 - \mathbb{P} \left[ \bigl|M^* - \mathbb{E}[M^*]\bigr| \geq \frac{\epsilon}{2} E_{\mathrm{max}} \right] \nonumber \\
    &\ge 1-2\binom{N}{k}^{-p\epsilon^2/4},\label{eq:second-part-TIS}
\end{align}
where~\eqref{eq:E111} follows by combining~\eqref{eq:E-of-M-star-lower} and~\eqref{eq:3eps/2}, whereas~\eqref{eq:second-part-TIS} follows by applying the second part of Theorem~\ref{thm:borell-TIS}, this time with 
\[
u=\frac{\epsilon\cdot E_{\mathrm{max}}}{2} \Rightarrow -\frac{u^2k^p}{2} = -\frac{p\epsilon^2}{4}\log\binom{N}{k}.
\]
Taking $\epsilon/2$ in place of $\epsilon$ yields that for all large $p$,
\begin{align}\label{eq:ground-state-bound-second-moment}
    \mathbb{P}\left[M^*\ge (1-\epsilon) E_{\mathrm{max}} \right]\ge 1-\binom{N}{k}^{-p\epsilon^2/16} \,.
\end{align}
Lastly, we combine~\eqref{eq:LargeT-Upper} and~\eqref{eq:ground-state-bound-second-moment}  to obtain
\begin{align}
    \mathbb{P} \Bigl[ E_{\mathrm{max}} \geq M^* \geq (1-\epsilon) E_{\mathrm{max}} \Bigr]    &=
    \mathbb{P}\left[M^*\ge (1-\epsilon) E_{\mathrm{max}} \right]
    +
    \mathbb{P} \left[M^*\le E_{\mathrm{max}} \right]
    -1
    \nonumber\\
    &\geq
    1-\binom{N}{k}^{- \frac{p\epsilon^2}{16}}
    - \Theta \left(\frac{1}{\sqrt{p \log \binom{N}{k}}} \right) \,.
\end{align}
This yields Theorem~\ref{thm:LargeTensor}. 
\subsection{Proof of Proposition~\ref{Prop:2ndMom}}\label{pf:PropMain}
Suppose $\alpha>0$ is such that $\frac{N}{k}\ge 1+\alpha>0$ for $k,N$ large. Such an $\alpha$ exists due to Assumption~\ref{Asm:Scaling}.

\paragraph{Second Moment Method} Recall the Paley-Zygmund inequality:
\begin{equation}\label{eq:PZ}
    \mathbb{P}[N_E\ge 1]\ge \frac{\mathbb{E}[N_E]^2}{\mathbb{E}[N_E^2]}.
\end{equation}
This is a standard consequence of Cauchy-Schwarz inequality, see, e.g.,~\cite{alon2016probabilistic}. In the remainder, we set $E=(1-\epsilon)E_{\mathrm{max}}$ and apply~\eqref{eq:PZ} to  $N_E$ in~\eqref{eq:N__E}. 

Repeating the first moment calculation for $N_E$, we find
\begin{align}
\mathbb{E}[N_E] &= (1+o_N(1))\binom{N}{k}^{p\epsilon(2-\epsilon)}\frac{1}{(1-\epsilon)\sqrt{4\pi p\log \binom{N}{k}}}\label{eq:N-E-general}\\
&=\binom{N}{k}^{p\epsilon(2-\epsilon)}\Theta\left(\frac{1}{\sqrt{p\log \binom{N}{k}}}\right) \label{eq:bound-E-NE-general-problem}.
\end{align}
Notice that $\mathbb{E}[N_E]=\binom{N}{k}^{\Theta(1)} = \omega(1)$.
\paragraph{A Counting Estimate}
We establish the following technical result to control certain counting terms appearing below.
\begin{lemma}\label{lemma:BinomCoeff}
    For any $\gamma>0$, there exists a $\delta>0$ such that
    \[
    \sum_{a>(1-\delta)k}\binom{k}{a}\binom{N-k}{k-a}\le \binom{N}{k}^\gamma.
    \]
\end{lemma}
\begin{proof}[Proof of Lemma~\ref{lemma:BinomCoeff}]
Recall the elementary estimates
\begin{equation}\label{eq:standard}
    \left(\frac{n}{k}\right)^k \le \binom{n}{k}\le \left(\frac{en}{k}\right)^k.
\end{equation}
    Notice that using $\binom{k}{a}\binom{N-k}{k-a}= \binom{k}{k-a}\binom{N-k}{k-a}$, 
    \begin{equation}\label{eq:SumEll}
           \sum_{a>(1-\delta)k} \binom{k}{a}\binom{N-k}{k-a} = \sum_{\ell<\delta k}\binom{k}{\ell}\binom{N-k}{\ell}.
        \end{equation}
{
Suppose that $\delta$ satisfies
\begin{equation}\label{eq:D-UpB}
    \delta < \min\left\{\frac12,\frac12\left(\frac{N}{k}-1\right),\gamma\right\} \quad\text{and}\quad \delta +2\delta \log\frac{e}{\delta} +\delta \log(1+\alpha)<\gamma \log(1+\alpha),
\end{equation}
where we recall $N,k$ scale such that $N/k\ge 1+\alpha$ for an $\alpha>0$ bounded away from zero as $N,k\to\infty$. Such a $\delta$ indeed exists since $\gamma\log(1+\alpha)>0$ and 
\[
\delta + 2\delta \log\frac{e}{\delta} + \delta\log(1+\alpha)\to 0
\] 
continuously as $\delta\to 0$. With $\delta$ chosen as in~\eqref{eq:D-UpB}, we thus have
\begin{equation}\label{eq:N-k-bd}
    \delta + 2\delta \log \frac{e}{\delta} + \delta\log \frac{N}{k}<\gamma\log\frac{N}{k}.
\end{equation}
and
\begin{align}
   \max_{\ell \le \delta k} \binom{k}{\ell} &\le \binom{k}{\delta k}\le \left(\frac{e}{\delta}\right)^{\delta k},\nonumber \\ 
   \max_{\ell\le \delta k}\binom{N-k}{\ell} &\le \binom{N-k}{\delta k}\le \left(\frac{eN}{\delta k}\right)^{\delta k}.\label{eq:Coeff2}
\end{align}
Using~\eqref{eq:SumEll}, we thus obtain
\begin{align}
  \sum_{\ell <\delta k }\binom{k}{\ell}\binom{N-k}{\ell} &\le \exp\left(\log(\delta k) + 2\delta k \log\frac{e}{\delta} + \delta k\log\frac{N}{k}\right)\label{eq:USE-Coeff2}    \\
  &\le \exp\left(k\left(\delta + 2\delta \log\frac{e}{\delta}+\delta \log\frac{N}{k}\right)\right) \label{eq:Large-k}\\
  &<\exp\left(k\gamma \log \frac{N}{k}\right)\le \binom{N}{k}^\gamma\label{eq:USE-eq:N-k-bd}
\end{align}
where~\eqref{eq:USE-Coeff2} uses~\eqref{eq:Coeff2},~\eqref{eq:Large-k} is valid for all $k\ge K(\delta)$, for $K(\delta)$ depending only on $\delta$,~\eqref{eq:USE-eq:N-k-bd} uses~\eqref{eq:N-k-bd} and~\eqref{eq:standard}. This yields Lemma~\ref{lemma:BinomCoeff}.
}
\end{proof}
We apply Lemma~\ref{lemma:BinomCoeff} with 
\begin{equation}\label{eq:Gamma-From-Lemma}
    \gamma = \frac{\epsilon(2-\epsilon)}{2}
\end{equation}
and $\delta$ prescribed accordingly. 
\paragraph{Second Moment Estimate: Covariance Profiles}
Observe that for $\Phi_{\mathcal{I}}$ in~\eqref{eq:N__E}
\begin{align}\label{eq:NEsq-largest-submatrix}
    \mathbb{E}
    \bigl[N_E^2\bigr]
    &=
\sum_{\mathcal{I},\mathcal{I'}\in\F}
    \mathbb{P}
    \Bigl[\Phi_{\mathcal{I}} > Ek^{p/2},
    \Phi_{\mathcal{I'}} > Ek^{p/2}
    \Bigr]
    \,.
\end{align}
For a given $\mathcal{I} = \left( I_{1}, \ldots, I_{p} \right)\in \F$ and $\mathcal{I'} = \left( I'_{1}, \ldots, I'_{p} \right)\in\F$ with $|I_i\cap I_i'| = a_i,\forall i$, let
\begin{align}\label{eq:LAMBDA}
    \lambda(a_{1},\ldots,a_{p})  := 
    \mathbb{E}
    \Bigl[\Phi_{\mathcal{I}}\cdot \Phi_{\mathcal{I}'}\Bigr]
    =
    \frac{1}{k^{p}}\prod_{i=1}^{p} a_i
    \,.
\end{align}
In what follows, we write $\lambda:=\lambda(a_1,\dots,a_p)$ for convenience.
Define
\begin{align}\label{eq:P-of-Lambda}
    p(\lambda) := \mathbb{P}\Bigl[\Phi_{\mathcal{I}} \ge Ek^{p/2}, \Phi_{\mathcal{I'}} \ge Ek^{p/2}\Bigr]
    \,,
\end{align}
where $\mathcal{I} = \left( I_{1}, \ldots, I_{p} \right)$ and $\mathcal{I'} = \left( I'_{1}, \ldots, I'_{p} \right) $ with $\left|I_{i} \cap I'_{i} \right| = a_{i}$. With this notation,~\eqref{eq:NEsq-largest-submatrix} becomes
\begin{align}\label{eq:NEsq-largest-submatrix-v2}
    \mathbb{E}\bigl[N_{E}^2\bigr] = \binom{N}{k}^p\sum_{0\le a_1,\dots,a_p\le k}\prod_{1\le i\le p}\binom{k}{a_i}\binom{N-k}{k-a_i} p(\lambda) \,.
\end{align}
To control~\eqref{eq:NEsq-largest-submatrix-v2}, we divide the sum into two terms. To this end, let
$\boldsymbol{a} = \left(a_{1},\ldots, a_{p} \right)$.
We partition the set $\{\boldsymbol{a}:a_i\in[0,k]\cap \mathbb{Z}\}$ into $J_1\cup J_2$ where
\begin{align}
    J_{1} &= \left\{ \boldsymbol{a}: \left|i : a_{i} \leq (1-\delta) k | \geq f(p) \right| \right\},\label{eq:J-1}\\
    J_{2} &= \left\{ \boldsymbol{a}: \left|i : a_{i} \leq (1-\delta) k | < f(p) \right| \right\} \label{eq:J-2}
\end{align}
for $f(p)=\epsilon p/2$ and $\delta$ defined below~\eqref{eq:Gamma-From-Lemma}.

Using~\eqref{eq:J-1} and~\eqref{eq:J-2}, we decompose
\begin{align}
    \mathbb{E} \bigl[ N_{E}^2 \bigr]
    =
    \underbrace{\binom{N}{k}^p\sum_{\boldsymbol{a} \in J_1}\prod_{1\le i\le p}\binom{k}{a_i}\binom{N-k}{k-a_i} p(\lambda)}_{:= S_{1}}
    +
    \underbrace{\binom{N}{k}^p\sum_{\boldsymbol{a} \in J_2}\prod_{1\le i\le p}\binom{k}{a_i}\binom{N-k}{k-a_i} p(\lambda)}_{:= S_{2}}
    \,.
\end{align}
We now establish a bivariate tail bound.
\begin{lemma}\label{lemma:p(a)}
 Set $\bar{E} = k^{p/2}E = k^{p/2}(1-\epsilon)E_{\mathrm{max}}$. Then,
    \begin{align*} 
           \frac{p(\lambda)}{\mathbb{P}\left[\cN(0,1)\ge \bar{E}\right]^2}&\le \frac{\left(1+\lambda\right)^2}{\sqrt{1-\lambda^2}} \exp\left(\lambda\bar{E}^2\right)\left(1+\frac{1}{\bar{E}^2}\right)^2 \\
    &=\bigl(1+o(1)\bigr)\frac{\left(1+\lambda\right)^2}{\sqrt{1-\lambda^2}} \exp\left(\lambda\bar{E}^2\right).
        \end{align*}
\end{lemma}
\begin{proof}[Proof of Lemma~\ref{lemma:p(a)}]
    We record the following tail bound from Nobel~\cite{BivTail}: for any $Z,Z_\rho\sim \cN(0,1)$ with $\mathbb{E}[Z\cdot Z_\rho]=\rho$ and $u>0$,
    \begin{equation}\label{eq:BIV_TAIL}
\mathbb{P}\bigl[Z>u,Z_\rho>u\bigr]\le \frac{(1+\rho)^2}{2\pi u^2\sqrt{1-\rho^2}}\exp\left(-\frac{u^2}{1+\rho}\right).
        \end{equation} 
    We use the last display with $Z=\Phi_{\mathcal{I}}$, $Z_\rho = \Phi_{\mathcal{I}'}$, $\rho=\lambda$ and $u=k^{p/2}E = k^{p/2}(1-\epsilon)E_{\mathrm{max}} = \bar{E}$ (via~\eqref{eq:Bar-of-E}). Recalling the lower tail bound~\eqref{eq:GaussTail} with $x=\bar{E}$ and noticing $\bar{E} = \omega(1)$, so that
    \[
    \left(1+\frac{1}{\bar{E}^2}\right)^2 = 1+o(1)
    \]
    we obtain Lemma~\ref{lemma:p(a)}.
\end{proof}
\paragraph{Controlling $S_1$}
Using Lemma~\ref{lemma:p(a)}, we control $S_1$. Observe that for $\boldsymbol{a} \in J_1$,~\eqref{eq:J-1} implies
\begin{align}
    \lambda = \frac{a_1\cdots a_p}{k^p}  \le \prod_{i:a_i\le (1-\delta)k}\frac{a_i}{k}\le (1-\delta)^{f(p)}.
\end{align}
Applying Lemma~\ref{lemma:p(a)}, we obtain
\begin{align}\label{eq:S1-bound}
    \frac{S_{1}}{\mathbb{E}[ N_{E} ]^2} &=
    \binom{N}{k}^{-p}\sum_{\boldsymbol{a} \in J_1}\prod_{1\le i\le p}\binom{k}{a_i}\binom{N-k}{k-a_i} \frac{p(\lambda)}{\mathbb{P}[\cN(0,1) > k^{p/2} E]^2}
    \nonumber\\
    &\leq 
    \frac{(1+(1-\delta)^{f(p)})^2}{\sqrt{1-(1-\delta)^{2f(p)}}}\exp\Bigl(\left(1-\delta\right)^{f(p)}\bar{E}^2\Bigr).
\end{align}
\paragraph{Controlling $S_2$}
Fix any $\boldsymbol{a}\in J_2$. The number of $i$ with $a_i\le (1-\delta)k$ is at most $f(p)$. Furthermore, the number of $i$ with $a_i>(1-\delta)k$ is trivially at most $p$. So, Lemma~\ref{lemma:BinomCoeff} gives
\begin{align*}
    \prod_{1\le i\le p}\binom{k}{a_i}\binom{N-k}{k-a_i} &=\left(\prod_{i:a_i\le (1-\delta)k}\binom{k}{a_i}\binom{N-k}{k-a_i}\right)\left(\prod_{i:a_i>(1-\delta)k}\binom{k}{a_i}\binom{N-k}{k-a_i}\right) \\
    &\le \binom{N}{k}^{f(p)}\binom{N}{k}^{\gamma p},
\end{align*}
implying via the bound 
\[
|J_2|\le \Bigl|[0,k]\cap \mathbb{Z}\Bigr|^p =(k+1)^p
\]
that
\begin{equation}\label{eq:COUNT-IMPORTANT}
    \sum_{\boldsymbol{a}\in J_2}\prod_{1\le i\le p}\binom{k}{a_i}\binom{N-k}{k-a_i}\le (k+1)^p  \binom{N}{k}^{f(p)}\binom{N}{k}^{\gamma p}
\end{equation}
This together with the trivial bound
$
p(\lambda)\le \mathbb{P}\bigl[\cN(0,1)\ge \bar{E}\bigr]$ yield
\begin{align}\label{eq:S2-bound-v1}
    \frac{S_{2}}{\mathbb{E}\left[ N_{E} \right]^2} &=
    \frac{1}{\mathbb{E}\left[N_{E}\right]^2} \binom{N}{k}^p \sum_{\boldsymbol{a} \in J_2}\prod_{1\le i\le p}\binom{k}{a_i}\binom{N-k}{k-a_i} p(\lambda)
    \nonumber\\
    &\leq
    \binom{N}{k}^p 
    \frac{\mathbb{P}\bigl[ \cN(0,1) \geq \bar{E} \bigr]}{\mathbb{E}\left[N_{E}\right]^2}\sum_{\boldsymbol{a} \in J_2}\prod_{1\le i\le p}\binom{k}{a_i}\binom{N-k}{k-a_i} 
    \nonumber\\
    &\leq 
    \frac{(k+1)^p}{\mathbb{E}[N_{E}]}
    \binom{N}{k}^{f(p)} \binom{N}{k}^{\gamma p},
\end{align}
where~\eqref{eq:S2-bound-v1} follows from~\eqref{eq:COUNT-IMPORTANT}. 
Combining~\eqref{eq:S2-bound-v1} with~\eqref{eq:bound-E-NE-general-problem}, we obtain
\begin{align}\label{eq:S2-bound-v2}
    \frac{S_{2}}{\mathbb{E}\left[ N_E\right]^2}
    \leq 
    \exp \left(\Bigl(f(p) + \gamma p -p\epsilon(2-\epsilon)\Bigr) \log \binom{N}{k} + p\log(k+1) \right)
    \Theta \left(\sqrt{p\log \binom{N}{k}} \right)
    \,.
\end{align}
Lastly, we note that 
\[
p\log(k+1) = o\left(\log\binom{N}{k}\right)
\]
which follows (a) for $k=o(N)$ from the expansion $\log\binom{N}{k} = (1+o(1))k\log \frac{N}{k}$ and (b) for $k=\Theta(N)$ from the fact $\log \binom{N}{k} = \Theta(N)$. Recalling $\gamma = \epsilon(2-\epsilon)/2$ from~\eqref{eq:Gamma-From-Lemma} and $f(p)=\epsilon p/2$, we arrive at
\begin{align}\label{eq:S2-bound}
    \frac{S_{2}}{\mathbb{E}\left[ N_E\right]^2}
    \leq 
    \exp \left( - \frac12 p\left(1-\epsilon\right)\epsilon \log \binom{N}{k} \right)
    \Theta \left[ \sqrt{p \log \binom{N}{k}} \right]
    \leq 
    \binom{N}{k}^{-\Theta(1)}
    \,.
\end{align}
We combine~\eqref{eq:S1-bound} and \eqref{eq:S2-bound} to obtain
\begin{align}\label{eq:controlling-second-estimate}
    \frac{\mathbb{E}\left[ N_{E}^2\right]}{\mathbb{E}\left[ N_{E}\right]^2}
    \leq
    \frac{(1+(1-\delta)^{f(p)})^2}{\sqrt{1-(1-\delta)^{2f(p)}}}\exp\Bigl(\left(1-\delta\right)^{f(p)}\bar{E}^2\Bigr)
    +
    \binom{N}{k}^{-\Theta(1)} \,.
\end{align}
Using the Paley-Zygmund Inequality~\eqref{eq:PZ} we conclude
\begin{align}
    \mathbb{P}\bigl[N_{E} \geq 1 \bigr]
    \ge
    \left(\frac{(1+(1-\delta)^{f(p)})^2}{\sqrt{1-(1-\delta)^{2f(p)}}}\exp\Bigl(\left(1-\delta\right)^{f(p)}\bar{E}^2\Bigr)
    +
    \binom{N}{k}^{-\Theta(1)}\right)^{-1},
\end{align}
where $f(p) = \epsilon p/2$ and the $\Theta(1)$ terms depends only on $p$ and $\epsilon$ due to~\eqref{eq:S2-bound}. This establishes Proposition~\ref{Prop:2ndMom}.
\subsection{Proof of Theorem~\ref{thm:M-OGP-generic-tensor-correlated}}\label{sec:proof-MOGP}
In this section, we provide a proof of Theorem~\ref{thm:M-OGP-generic-tensor-correlated}.
The proof relies on the first moment method. We first define a set to count the tuples $\mathcal{I}_{i}$ of multi-indices that satisfy conditions $\mathrm{(a)}$ and $\mathrm{(b)}$  of Definition~\ref{Def:Overlap-Set-correlated-generic-subtensor}.
\begin{align}
    F(m,\nu_1,\nu_2) := 
    \left\{ \left(\mathcal{I}_{1},\ldots, \mathcal{I}_{m}\right):
    \frac{\left|I_{n,i_1} \cap I_{n,i_2} \right|}{k} \in \left[\nu_1,\nu_2\right],\,
    \forall n\in[p],1\le i_1<i_2\le m
    \right\}
    \,.
\end{align}
We now observe that
\begin{align}
    \bigl|S_{\gamma}\left(m, \nu_1, \nu_2, \mathcal{J} \right) \bigr|
    &=
    \sum_{\left(\mathcal{I}_{1},\ldots, \mathcal{I}_{m}\right)
    \in 
    F(m,\nu_1,\nu_2)}
    \mathbbm{1}
    \bigg\{
    \exists
    \tau_{1},\ldots,\tau_m \in \mathcal{J}
    :
    \\
    &
    \min_{1 \leq \ell \leq m}
    \frac{1}{k^{p/2}} \sum_{i_1\in I_{i,\ell},\dots,i_p\in I_{p,\ell}} \hat{A}^{(\ell)}_{i_1,\dots,i_p} (\tau_\ell)\ge \gamma k^{p/2} E_{\mathrm{max}}
    \bigg\}
    \,.
\end{align}
In order to apply the first moment method, we need to bound $\mathbb{E}\bigl[\bigl|S_{\gamma}\left(m, \nu_1, \nu_2, \mathcal{J} \right) \bigr| \bigr]$. We have:
\begin{align}\label{eq:upper-bound-size-of-S}
    &\mathbb{E}\left[\left|S_{\gamma}\left(m, \nu_1, \nu_2, \mathcal{J} \right) \right| \right]
    \nonumber
    \\ &=
    \sum_{\left(\mathcal{I}_{1},\ldots, \mathcal{I}_{m}\right)
    \in 
    F(m,\nu_1,\nu_2)}
    \mathbb{P}
    \bigg[\exists
    \tau_{1},\ldots,\tau_m \in \mathcal{J}
    :
    \min_{1 \leq \ell \leq m}
    \frac{1}{k^{p/2}} \sum_{i_1\in I_{i,\ell},\dots,i_p\in I_{p,\ell}} \hat{A}^{(\ell)}_{i_1,\dots,i_p} (\tau_\ell)\ge \gamma k^{p/2} E_{\mathrm{max}} \bigg]
    \nonumber
    \\
    &\leq
    \left|F(m,\nu_1,\nu_2)\right| \sup_{(\mathcal{I}_1,\dots,\mathcal{I}_m)\in F(m,\nu_1,\nu_2)}
    \mathbb{P}
    \bigg[\exists
    \tau_{1},\ldots,\tau_m \in \mathcal{J}
    : \nonumber \\
    &\hspace{5cm}
    \min_{1 \leq \ell \leq m}
    \frac{1}{k^{p/2}} \sum_{i_1\in I_{i,\ell},\dots,i_p\in I_{p,\ell}} \hat{A}^{(\ell)}_{i_1,\dots,i_p} (\tau_\ell)\ge \gamma k^{p/2} E_{\mathrm{max}} \bigg]
    \nonumber
    \\
    &\leq
    \left|F(m,\nu_1,\nu_2)\right| \binom{N}{k}^{cm} \sup_{(\mathcal{I}_1,\dots,\mathcal{I}_m)\in F(m,\nu_1,\nu_2)}
    \sup_{\tau_1,\ldots,\tau_m \in [0,\pi/2]}
    \mathbb{P}\left[\hat{\Phi}_{\mathcal{I}_i} > k^{p/2} \gamma E_{\mathrm{max}},1\le i\le m\right]
\end{align}
where
\begin{equation}\label{hat-PHI}
    \hat{\Phi}_{\mathcal{I}_{\ell}} := \frac{1}{k^{p/2}} \sum_{i_1\in I_{1,\ell},\dots,i_p\in I_{p,\ell}} \hat{A}^{(\ell)}_{i_1,\dots,i_p} (\tau_\ell)
    \,.
\end{equation}
Above,~\eqref{eq:upper-bound-size-of-S} follows by taking a union bound over all $\tau_1,\dots,\tau_m\in\mathcal{J}$, together with the fact that $|\mathcal{J}|\le \binom{N}{k}^{c}$.  Note that the bound in~\eqref{eq:upper-bound-size-of-S} is valid for all $c,\nu_1,\nu_2$ and $\gamma$.

We next bound the quantities
\begin{align*}
    C_0 := \left|F(m,\nu_1,\nu_2)\right|,\quad\text{and}\quad 
    C_1 := \sup_{\substack{(\mathcal{I}_1,\dots,\mathcal{I}_m)\in F(m,\nu_1,\nu_2)\\\tau_1,\ldots,\tau_m \in [0,\pi/2]}}
    \mathbb{P}\left[\hat{\Phi}_{\mathcal{I}_i} > k^{p/2} \gamma E_{\mathrm{max}},1\le i\le m\right] \,.
\end{align*}
\begin{proposition}\label{proposition:C0-C1-bounds}
    For sufficiently small $\eta := \nu_2 - \nu_1$, the following bounds hold
    \begin{align}
        C_0 &\leq \exp_2\left( p\log_2 \binom{N}{k} + m p\log_2 k + mp\left(kh(\nu_1) + k(1-\nu_1) \log_2\frac{eN}{k(1-\nu_1)}\right)\right) \,,\\
        C_1 &\leq \exp_2\left[ 
    -\gamma^2 p \log_2 \binom{N}{k} \frac{m}{1+ 2 m p \nu_2^p}
    \left(1
    +
    o_N(1)\right)
    \right]\,.
    \end{align}
\end{proposition}
We prove Proposition~\ref{proposition:C0-C1-bounds} in Sections~\ref{sec:bound-C0} and \ref{sec:bound-C1}.
Assuming the validity of Proposition~\ref{proposition:C0-C1-bounds}, we establish using~\eqref{eq:upper-bound-size-of-S}
\begin{align}
    \mathbb{E}\left[\left|S_{\gamma}\left(m, \nu_1, \nu_2, \mathcal{J} \right) \right| \right]
    \leq
    \exp_2 \left[ p \log_2 \binom{N}{k} \Psi(c,m,\nu_1,\nu_2) + O
    \left(
    \log \log \binom{N}{k}
    \right)\right]
\end{align}
where
\begin{align}
    \Psi
    =
    1  + m\bigg[c - \frac{ \gamma^2 }{1+ 2 m p \nu_2^p}
    + \frac{k}{\log_2 \binom{N}{k}}h(\nu_1) + \frac{k}{\log_2 \binom{N}{k}}(1-\nu_1) \log_2\frac{eN}{k(1-\nu_1)}
    \bigg].
\end{align}
We now show that $\Psi$ is negative, allowing us to establish:
\begin{align*}
    \mathbb{E}\left[\left|S_{\gamma}\left(m, \nu_1, \nu_2, \mathcal{J} \right) \right| \right]
    \leq
    \exp\left[ - \Theta\left(p \log_2 \binom{N}{k}\right) \right] \,.
\end{align*}
Define $\delta := m \gamma^2 - 1 >0$. We split the analysis into two cases:
\subsubsection*{Case 1: $k = o(N)$}
We have
\begin{align}
    \Psi
    &=
    1  + m\bigg(c  - \frac{ \gamma^2 }{1+ 2 m p \nu_2^p}
    \bigg)
    +
    o\left(1\right)
    \,.
\end{align}
Choose $P^*$ such that for all $p \geq P^*$,
\begin{align}
    1 - \frac{ m\gamma^2 }{1+ 2 m p \nu_2^p} \leq -\frac{\delta}{2}.
\end{align}
Finally, we choose $c \leq \frac{\delta}{4m}$.
These choices collectively lead to
\begin{align*}
    \Psi
    &\leq 
    -\frac{\delta}{4} + o\left(1\right) <0 \,, 
\end{align*}
for sufficiently large $k$ and $N$.
\subsubsection*{Case 2: $k = \alpha N, \alpha >0$}
In this case, we have
\begin{align}
    \Psi
    &=
    1 + 
    cm - \frac{\gamma^2 m}{1 + 2m p \nu_2^p}
    +
    \frac{\alpha}{h(\alpha)} h(\nu_1) + \frac{\alpha}{h(\alpha)}(1 - \nu_1) \log_2 \frac{e }{\alpha (1- \nu_1)}
    +
    o(1)\,.
\end{align}
Let us first choose $\nu_1$ sufficiently close to 1, so that
\begin{align}
    \alpha h(\nu_1) + \alpha(1 - \nu_1) \log_2 \frac{e }{\alpha (1- \nu_1)}
    \leq \frac{\delta h(\alpha)}{8 }.
\end{align}
Notice that such a choice indeed exists since $\alpha$ is fixed, whereas
\[
  \alpha h(\nu_1) + \alpha(1 - \nu_1) \log_2 \frac{e }{\alpha (1- \nu_1)}\to 0,\quad\text{as}\quad \nu_1\to 1.
\]
Choose $P^*$ such that for all $p \geq P^*$, 
\begin{align}
    1 - \frac{m\gamma^2}{1+ 2 m p \nu_2^p} \leq -\frac{\delta}{2}.
\end{align}
Finally, we choose $c \leq \frac{\delta}{8m}$. 
These choices collectively lead to
\begin{align*}
    \Psi
    &\leq -\frac{\delta}{4} + o(1) <0 \,, 
\end{align*}
for sufficiently large $k$ and $N$. We thus conclude that
\begin{align*}
    \mathbb{E}\left[\left|S_{\gamma}\left(m, \nu_1, \nu_2, \mathcal{J} \right) \right| \right]
    \leq
    \exp\left[ - \Theta\left(p \log_2 \binom{N}{k}\right) \right]\,.
\end{align*}
Using Markov's inequality we obtain
\begin{align*}
    \mathbb{P} \left[S_{\gamma}\left(m, \nu_1, \nu_2, \mathcal{J} \right) = \varnothing \right]
    \geq 
    1 - \mathbb{E}\left[\left|S_{\gamma}\left(m, \nu_1, \nu_2, \mathcal{J} \right) \right| \right]
    \geq
    1 - \exp\left[ - \Theta\left(p \log \binom{N}{k}\right) \right]
    \,.
\end{align*}
This yields Theorem~\ref{thm:M-OGP-generic-tensor-correlated}. It thus suffices to prove Proposition~\ref{proposition:C0-C1-bounds}.
\subsection*{Controlling $C_0$ }\label{sec:bound-C0}
In this section, we prove the first part of Proposition~\ref{proposition:C0-C1-bounds}. Let us first observe a simple combinatorial bound:
\begin{equation}\label{eq:F_COUNT}
        C_0 \le \binom{N}{k}^p\left(\sum_{\nu_1\le \alpha_i\le \nu_2 , \alpha_i k\in\mathbb{N}} \prod_{1 \leq i \leq p}\binom{k}{\alpha_i k}\binom{N-k}{k-\alpha_i k}\right)^{m-1}
\end{equation}
We choose $\nu_1 = \frac{R-1}{R}\nu_2$ for a sufficiently large integer $R$ and $\nu_2$ close to 1 to ensure the binomial coefficients are not vacuous. In particular, the assumption $\nu_1>\frac12$ holds.

We recall the standard estimates
\[
\max_{\nu_1\le \alpha\le \nu_2}\binom{k}{\alpha k}\le \exp_2\bigl(kh(\nu_1)\bigr)
\]
and 
\[
\max_{\nu_1\le \alpha\le \nu_2}\binom{N-k}{k-\alpha k} \le \left(\frac{e(N-k)}{k(1-\nu_1)}\right)^{k(1-\nu_1)} \le \exp_2\left(k(1-\nu_1)\log_2\frac{eN}{k(1-\nu_1)}\right).
\]
Combining the last two displays, we bound~\eqref{eq:F_COUNT} by 
\begin{align}\label{eq:counting-m-ogp}
    &C_0 \le \binom{N}{k}^p\left(\sum_{\nu_1\le \alpha_i\le \nu_2 , \alpha_i k\in\mathbb{N}}\exp_2\left(pkh(\nu_1) + pk(1-\nu_1) \log_2\frac{eN}{k(1-\nu_1)}\right)\right)^{m-1} \nonumber\\ 
    &
    \hspace{2cm}
    \le \exp_2\left( p\log_2 \binom{N}{k} + m p\log_2 k + mp\left(kh(\nu_1) + k(1-\nu_1) \log_2\frac{eN}{k(1-\nu_1)}\right)\right),
\end{align}
where the last step uses the fact 
\[
|\alpha:\nu_1\le \alpha\le \nu_2,k\alpha\in\mathbb{N}|\le k.
\]
This proves part $\mathrm{(a)}$ of Proposition~\ref{proposition:C0-C1-bounds}.
\subsection*{Controlling $C_1$}\label{sec:bound-C1}
In this section, we prove the second part of Proposition~\ref{proposition:C0-C1-bounds}.
\subsubsection*{Auxiliary Results from linear algebra and probability theory}\label{sec:Auxiliary-results}
We begin by recording auxiliary results from linear algebra and probability, including the Sherman-Morrison formula matrix inversion formula~\cite{sherman1950adjustment} and Wielandt-Hoffman inequality~\cite{hoffman1953variation}; multivariate normal tail bounds~\cite{savage1962mills,hashorva2003multivariate,hashorva2005asymptotics}, and Slepian's Gaussian comparison lemma~\cite{slepian1962one}.
\begin{theorem}[Sherman-Morrison formula]\label{thm:sm}
Let $A\in\R^{n\times n}$ be an invertible matrix and $u,v\in\R^n$ be column vectors. Then, $(A+uv^T)^{-1}$ exists iff $1+v^T A^{-1}u\ne 0$ and the inverse is given by the formula
\[
\bigl(A+uv^T\bigr)^{-1} = A^{-1} - \frac{A^{-1}uv^T A^{-1}}{1+v^T A^{-1}u}.
\]
\end{theorem}
\begin{theorem}[Wielandt-Hoffman inequality]\label{thm:hw}
Let $A,A+E\in\R^{n\times n}$  be two symmetric matrices with respective eigenvalues
\[
\lambda_1(A)\ge \lambda_2(A)\ge\cdots\ge\lambda_n(A)\quad\text{and}\quad \lambda_1(A+E)\ge \lambda_2(A+E)\ge\cdots\ge \lambda_n(A+E).
\]
Then
\[
\sum_{1\le i\le n}\left(\lambda_i(A+E)-\lambda_i(A)\right)^2\le \|E\|_F^2.
\]
\end{theorem}
\begin{theorem}[Multivariate Gaussian tail bound] \label{thm:multiv-tail}
Let $\boldsymbol{X}\in\R^d$ be a centered multivariate normal random vector with non-singular covariance matrix $\Sigma\in\R^{d\times d}$ and $\boldsymbol{t}\in\R^d$ be a fixed threshold. Suppose that $\Sigma^{-1}\boldsymbol{t}>\boldsymbol{0}$ entrywise. Then,
\[
1-\ip{1/(\Sigma^{-1}\boldsymbol{t})}{\Sigma^{-1}(1/(\Sigma^{-1}\boldsymbol{t})} \le \frac{\mathbb{P}[\boldsymbol{X}\ge \boldsymbol{t}]}{\varphi_{\boldsymbol{X}}(\boldsymbol{t})\prod_{i\le d}\ip{e_i}{\Sigma^{-1}\boldsymbol{t}}}\le 1,
\]
where $e_i\in\R^d$ is the $i{\mathrm{th}}$ unit vector and $\varphi_{\boldsymbol{X}}(\boldsymbol{t})$ is the multivariate normal density evaluated at $\boldsymbol{t}$:
\[
\varphi_{\boldsymbol{X}}(\boldsymbol{t}) = (2\pi)^{-d/2}|\Sigma|^{-1/2}\exp\left(-\frac{\boldsymbol{t}^T\Sigma^{-1}\boldsymbol{t}}{2}\right)\in\R^+.
\]
\end{theorem}
\begin{theorem}[Slepian's lemma]\label{lemma:Slepian}
Let $X = (X_1, \ldots , X_n)$and $Y = (Y_1,\ldots, Y_n)$ be multivariate normal random vectors such that $\mathbb{E}[X_i] = \mathbb{E}[Y_i] = 0, \forall i$, $\mathbb{E}[X_i^2] = \mathbb{E}[Y_i^2], \forall i$, and
\begin{align*}
    \mathbb{E}\left[X_{i} X_{j}\right]
    \leq \mathbb{E}\left[Y_{i} Y_{j}\right],
    \quad 1\leq i< j \leq n.
\end{align*}
Fix any $c_{1},\ldots, c_{n} \in \mathbb{R}$.
Then,
\begin{align*}
    \mathbb{P}\left[ X_i \leq c_i,\, \forall \, i \right]
    \leq 
    \mathbb{P}\left[ Y_i \leq c_i, \,\forall \, i \right] ,
    \quad 
    \mathbb{P}\left[ X_i \geq c_i, \, \forall \, i \right]
    \leq 
    \mathbb{P}\left[ Y_i \geq c_i, \, \forall \, i \right]
    .
\end{align*}
\end{theorem}
\subsection*{Proof of the Second Part of Proposition~\ref{proposition:C0-C1-bounds}}\label{sec:proof-part-b-proposition}
We now bound $C_1$. Fix $(\mathcal{I}_1,\dots,\mathcal{I}_m)\in F(m,\nu_1,\nu_2)$ and $\tau_1,\dots,\tau_m \in \mathcal{J}$. Furthermore, define $ \eta(i,j,q)$ such that
\begin{equation}\label{eq:INTERSECT_PATTERN}
   \frac{|I_{q,i}\cap I_{q,j}|}{k} = \nu_2 - \eta(i,j,q) \ge \nu_1,
\end{equation}
for every $1\le i<j\le m$ and $1\le q\le p$. 
For fixed $\tau_{1},\ldots,\tau_{m}$,
we observe that for $1\leq i <j \leq m$
\begin{align*}
\mathbb{E}\left[\hat{\Phi}_{\mathcal{I}_{i}} \hat{\Phi}_{\mathcal{I}_{j}}\right]
    &=
    \frac{1}{k^p}
    \sum_{i_1\in I_{1,i},\dots,i_p\in I_{p,i}}
    \sum_{j_1\in I_{1,j},\dots,j_p\in I_{p,j}}
    \mathbb{E}
    \left[ 
    A^{(i)}_{i_1,\dots,i_p}(\tau_i)
    A^{(j)}_{j_1,\dots,j_p}(\tau_j)
    \right] \\
    &=
    \frac{1}{k^p}
    \sum_{i_1\in I_{1,i},\dots,i_p\in I_{p,i}}
    \sum_{j_1\in I_{1,j},\dots,j_p\in I_{p,j}}
    \mathbb{E}\bigg[\left( \cos(\tau_i)
            A^{(0)}_{i_1,\dots,i_p}
            +
            \sin(\tau_i)
            A^{(i)}_{i_1,\dots,i_p}\right)
            \\
            &
            \hspace{6cm}
            \times\left(\cos(\tau_j)
            A^{(0)}_{j_1,\dots,j_p}
            +
            \sin(\tau_j)
            A^{(j)}_{j_1,\dots,j_p} \right) \bigg]
    \\
    &=
    \frac{\cos(\tau_i) \cos(\tau_j)}{k^p}
    \sum_{i_1\in I_{1,i}\cap I_{1,j},\dots,i_p\in I_{p,i}\cap I_{p,j}} 1
    \\
    &
    =
    {\cos(\tau_i) \cos(\tau_j)} \prod_{q=1}^{p}\left(\nu_2 - \eta(i,j,q) \right).
\end{align*}
For $\Phi_{\mathcal{I}}$ defined in~\eqref{eq:INFORMAL}, we thus obtain that
\begin{align}
    \mathbb{E}\left[\hat{\Phi}_{\mathcal{I}_i} \hat{\Phi}_{\mathcal{I}_j} \right]
    \leq 
    \mathbb{E}\left[{\Phi}_{\mathcal{I}_i} {\Phi}_{\mathcal{I}_j} \right]
    \,,
\end{align}
Applying now Slepian's lemma, Theorem~\ref{lemma:Slepian}, we arrive at the bound
\begin{align}\label{eq:sup-bound-slepian}
    \sup_{\tau_1,\ldots,\tau_m \in [0,\pi/2]}
    \mathbb{P}\left[\hat{\Phi}_{\mathcal{I}_i} > k^{p/2} \gamma E_{\mathrm{max}},\, 
 \forall\, i\right]
    \leq 
    \mathbb{P}\left[\Phi_{\mathcal{I}_i} > k^{p/2} \gamma E_{\mathrm{max}},\, 
 \forall\, i\right]
    \,.
\end{align}
That is,  to control $C_1$, we restrict our attention to analyzing the tails of $\Phi_{\mathcal{I}_i}$, i.e., the case $\tau_1 = \ldots = \tau_m = 0$.
We note that $\left(\Phi_{\mathcal{I}_1},\ldots,\Phi_{\mathcal{I}_m}\right)$ is a multivariate Gaussian with
\begin{align}\label{eq:momnets-prop-Phi-I-i}
&\mathbb{E}\left[\Phi_{\mathcal{I}_i}\right] = 0 \,, 
    \mathbb{E}\left[\Phi_{\mathcal{I}_i}^2\right] = 1\,,
    \, 
     \nonumber\\ &\mathbb{E}\left[\Phi_{\mathcal{I}_i} \Phi_{\mathcal{I}_j}\right] = \prod_{q=1}^{p}\left(\nu_2 - \eta(i,j,q) \right),
\end{align}
where 
\[
\eta(i,j,q) \in \left[0, \nu_2 - \nu_1 \right],\quad\text{for}\quad  
    1\leq i <j \leq m \quad \text{and}\quad 1\le q\le p.
\]
We now use Theorem~\ref{thm:multiv-tail} to bound $\mathbb{P}\left[\Phi_{\mathcal{I}_i} > k^{p/2} \gamma E_{\mathrm{max}},\, \forall\, i\right]$. 
Define
\begin{align*}
    \boldsymbol{X} = \left(\Phi_{\mathcal{I}_1},\ldots,\Phi_{\mathcal{I}_m}\right) \,, \quad
    \boldsymbol{t} = k^{p/2} \gamma E_{\mathrm{max}} \boldsymbol{\mathbbm{1}}\,.
\end{align*}
\paragraph{Covariance of $\boldsymbol{X}$} Let $\Sigma\in\R^{m\times m}$ be the covariance matrix of $\boldsymbol{X}$. Notice that $\Sigma_{ii}=1$ for all $1\le i\le m$ and for all $1\le i<j\le m$,
\begin{equation}\label{eq:covariance-matrix-def}
    \Sigma_{ij} = \Sigma_{ji} = \prod_{q=1}^p\bigl(\nu_2 - \eta(i,j,q)\bigr)\,.
\end{equation}
Define auxiliary matrices $\Sigma_0,E\in\R^{m\times m}$ as follows.
\begin{equation}\label{eq:Sigma0ij-def} 
   \Sigma_{0} = \left(1-\nu_2^p\right) I_{m} + \nu_2^p \boldsymbol{1} \boldsymbol{1}^T, 
\end{equation}
where $I_m$ is the $m\times m$ identity matrix and $\boldsymbol{1}\in \R^m$ is the all-ones vector. As for the $E$, we set
\begin{align}
    E_{ii} &= 0,\quad \forall 1 \leq i \leq m \nonumber \\ 
    \label{eq:Eij-def}
    E_{ij} &= E_{ji} = \nu_2^p - \prod_{q=1}^{p} \left( \nu_2 - \eta(i,j,q)\right),\quad 1\le i<j\le m.
\end{align}
Notice that
\begin{align*}
    \Sigma = \Sigma_{0} - E \,.
\end{align*}
We now establish some technical results regarding $\Sigma, \Sigma_{0}$ and $E$ that will let us apply Theorem~\ref{thm:multiv-tail}.
These results are closely related to Lemma 3.13 of~\cite{gamarnik2023shatteringisingpurepspin}.
\begin{lemma}\label{lemma:auxilary-results}
    Let $\eta = \nu_2 - \nu_1$ and consider the matrices $\Sigma,\Sigma_{0}$ and $E$ defined in \eqref{eq:covariance-matrix-def}, \eqref{eq:Sigma0ij-def} and \eqref{eq:Eij-def}.
    \begin{itemize}
        \item [(a)] The matrix elements $E_{k \ell}$ are bounded by
        \begin{align*}
            0 \leq E_{k \ell} \leq p \eta \nu_2^{p-1}, \quad 1 \leq k \leq \ell \leq m\,.
        \end{align*}
        \item [(b)] The eigenvalues of the matrix $\Sigma_{0}$ are $1-\nu_2^p$ with multiplicity $m-1$ and $1 + (m-1) \nu_2^p$.
        \item [(c)] The determinant and inverse of $\Sigma_{0}$ are given by
        \begin{align*}
            \left| \Sigma_{0} \right| &= \left(1-\nu_2^p\right)^{m-1} \left(1+ (m-1) \nu_2^p\right)\,\\
            \Sigma_{0}^{-1} &= \frac{1}{1-\nu_2^p} I - \frac{\nu_2^p}{(1-\nu_2^p)(1+(m-1) \nu_2^p)} \boldsymbol{1} \boldsymbol{1}^T \,.
        \end{align*}
        \item [(d)] The covariance matrix $\Sigma$ is positive definite if 
        \begin{align}\label{eq:eta-cond}
            \eta <\frac{1-\nu_2^p}{m p \nu_2^{p-1}}\,.
        \end{align}
        \item [(e)] For sufficiently small $\eta$, the vector $\Sigma^{-1} \boldsymbol{1}$ is entry-wise positive. 
        \item [(f)] We have the following bounds for the determinant and inner product of the covariance matrix with a unit vector $e_i$
        \begin{align*}
            &\left|\Sigma\right|^{-1/2} \leq \left(1- 2 m p \nu_2^p\right)^{-m/2} = O_N(1)\,,\\
            &\prod_{i\leq m } \left<e_i, \Sigma^{-1} \boldsymbol{1} \right>
            \leq m^{m/2} \left(1- 2 m p \nu_2^p\right)^{-m} = O_{N}(1) \,.
        \end{align*}
        \item [(g)] The product $\boldsymbol{1}^T \Sigma^{-1} \boldsymbol{1}$ satisfies the inequality
        \begin{align*}
            \boldsymbol{1}^T \Sigma^{-1} \boldsymbol{1} \geq \frac{m}{1+ 2 m p \nu_2^p} \,.
        \end{align*}
    \end{itemize}
\end{lemma}
\begin{proof}
Before we proceed with the proof, let us setup some notation. We use $\Lambda_{1} \leq \ldots \leq \Lambda_{m} $ to denote the eigenvalues of $\Sigma$ and $\lambda_{1} \leq \ldots \leq \lambda_{m} $ to denote the eigenvalues of $\Sigma_{0}$.
    \begin{itemize}
        \item [(a)]
        From the definition of $E_{ij}$~\eqref{eq:Eij-def}, we see that $E_{ii} = 0$. We recall that $\eta(i,j,q) \in [0,\eta ]$.
        We bound $E_{ij}$ as follows. 
        \begin{align*}
            &0 \leq E_{ij} = \nu_2^p - \prod_{q=1}^{p} \left(\nu_2 - \eta(i,j,q)\right)
            \leq 
            \nu_2^p - \left(\nu_2 - \eta\right)^p
            \\
            &
            = \eta \sum_{0\leq q \leq p-1} \nu_2^{p-1-q} \left(\nu_2 -\eta \right)^{q} \leq \nu_2^{p-1}p \eta \,, \quad i\neq j \,.
        \end{align*}
        \item [(b)]
        One can easily check that $\boldsymbol{1}$ is an eigenvector of $\Sigma_{0}$ with eigenvalue $1 + (m-1) \nu_2^p$.
        Likewise, there are $m-1$ eigenvectors perpendicular to $\boldsymbol{1}$, corresponding to eigenvalue $1 - \nu_2^p$.
        \item  [(c)]
        The formula for the determinant directly follows from part (b).
        To obtain the inverse, we use the Sherman-Morrison formula (Theorem~\ref{thm:sm}) with 
        \begin{align*}
            A = \left(1-\nu_2^p\right) I_{m}\,,
            \quad
            u = v = \nu_2^{p/2}\boldsymbol{1}\,.
        \end{align*}
        \item [(d)] From part (b), we know that the smallest eigenvalue of $\Sigma_{0}$ is $\lambda_{1} = 1 - \nu_2^p$.
        If we can show that the smallest eigenvalue of $\Sigma$ is positive, we have the result. We first note that part (a) implies that
        \begin{align*}
            \left\Vert E \right\Vert_{2}
            \leq 
            \left\Vert E \right\Vert_{F}
            \leq m \eta p \nu_2^{p-1}\,.
        \end{align*}
        We now use the Wielandt-Hoffman inequality (Theorem~\ref{thm:hw})
        \begin{align}\label{eq:wh-on-Lambda-1}
            \bigl| \Lambda_{1} - (1 - \nu_2^p) \bigr| \leq \left\Vert E \right\Vert_{F}
            \leq m \eta p \nu_2^{p-1}
            \implies  
            \Lambda_{1} \geq (1 - \nu_2^p) - m \eta p \nu_2^{p-1}\,.
        \end{align}
        We easily see that
        \begin{align}\label{eq:Lambda_1_ineq}
           \eta <\frac{1-\nu_2^p}{m p \nu_2^{p-1}} \implies  \Lambda_{1} \geq (1 - \nu_2^p) - m \eta p \nu_2^{p-1} > 0.
        \end{align}
        \item [(e)] Define
        \begin{align}
            \boldsymbol{\eta} = \left\{\eta(i,j,q), 1 \le q \le p, 1\le i <j \le m \right\}\,.
        \end{align}
       Note that the map $f_i:\boldsymbol{\eta} \to \left( \Sigma^{-1} \boldsymbol{t} \right)_i$, where $\left( \Sigma^{-1} \boldsymbol{t} \right)_i$ is the $i{\mathrm{th}}$ entry of $\Sigma^{-1} \boldsymbol{t}$, is continuous and that \[
       \boldsymbol{\eta} \in \left[0, \nu_2 -\nu_1\right]^{pm(m-1)/2}
       \]
       which is a compact domain.
        Therefore, if we verify that $f_i(\boldsymbol{0}) = \left(\Sigma_{0}^{-1} \boldsymbol{1}\right)_i >0$, then the proof follows by uniform continuity. Using part (c),
        \begin{align*}
            \Sigma_{0}^{-1} \boldsymbol{1} = \frac{1}{1-\nu_2^p + m \nu_2^p} \boldsymbol{1}
        \end{align*}
        which is indeed entry-wise positive.
        \item [(f)] From~\eqref{eq:Lambda_1_ineq}, we see that
        \begin{align*}
            \Lambda_{1} \geq (1 - \nu_2^p) - m \eta p \nu_2^{p-1} > 1 - (mp + 1) \nu_2^p > 1- 2 m p \nu_2^p \,.
        \end{align*}
        Therefore, we see that
        \begin{align*}
            \left|\Sigma\right|^{-1/2}
            =
            \prod_{1\le i \le m} \Lambda_{i}^{-1/2}
            \leq \Lambda_{1}^{-m/2} 
            \leq \left(1- 2 m p \nu_2^p\right)^{-m/2}\,.
        \end{align*}
        Next, using Cauchy-Schwarz inequality, we obtain
        \begin{align*}
            \left<e_{i}, \Sigma^{-1} \boldsymbol{1} \right>
            \leq 
            \left\Vert e_{i} \right\Vert_{2}
            \left\Vert \Sigma^{-1} \boldsymbol{1} \right\Vert_{2}
            \leq 
            \left\Vert \Sigma^{-1}  \right\Vert_{2}
            \left\Vert \boldsymbol{1}  \right\Vert_{2}
            \leq 
            \frac{\sqrt{m}}{1 - 2mp \nu_2^p}\,.
        \end{align*}
        Thus,
        \begin{align*}
            \prod_{i\leq m } \left<e_i, \Sigma^{-1} \boldsymbol{1} \right>
            \leq m^{m/2} \left(1- 2 m p \nu_2^p\right)^{-m} \,.
        \end{align*}
        \item [(g)]
        We diagonalize $\Sigma^{-1}$ as
        \begin{align*}
            \Sigma^{-1} &= Q^T \Lambda Q \,,
        \end{align*}
        where $Q$ is an orthogonal matrix.
        We let $v := Q \boldsymbol{1}$ and note that
        \begin{align}\label{eq:1T-Sigma-1-bound}
            \boldsymbol{1}^T \Sigma^{-1} \boldsymbol{1}
            =
            \sum_{1 \leq i \leq m} 
            \Lambda_{i}^{-1} v_{i}^2
            \geq 
            \frac{1}{\Lambda_{m}}
            \sum_{1 \leq i \leq m} v_{i}^2
            =
            \frac{1}{\Lambda_{m}} \left\Vert Q \boldsymbol{1} \right\Vert_{2}^2
            =
            \frac{m}{\Lambda_{m}}
            \,.
        \end{align}
        We now apply Wielandt-Hoffman inequality to conclude that
        \begin{align}\label{eq:Lambda_m_bound}
            \left| \Lambda_{m} - \left(1 + (m-1) \nu_2^p \right) \right|
            \leq
            \left\Vert E \right\Vert_{F}
            \leq m p \eta \nu_2^{p-1} \implies \Lambda_{m} \leq 1 + 2 m p \nu_2^p \,.
        \end{align}
        Using~\eqref{eq:Lambda_m_bound} in \eqref{eq:1T-Sigma-1-bound}, we complete the proof of Lemma~\ref{lemma:auxilary-results}.
    \end{itemize}
\end{proof}
We can now apply Theorem~\ref{thm:multiv-tail}.
We use Lemma~\ref{lemma:auxilary-results} to conclude that
\begin{align*}
    &\boldsymbol{t}^T \Sigma^{-1} \boldsymbol{t}
    \geq
    \frac{2 \gamma^2 p m}{1 + 2 m p \nu_2^p} \log \binom{N}{k}\,,\\
    &\left|\Sigma\right|^{-1/2}
    = O_{N}(1) \,,
    \quad 
    \prod_{i\leq m} \left<e_i, \Sigma^{-1} \boldsymbol{t} \right>
    = O_{N}(1) \,.
\end{align*}
Let us combine these estimates and use Theorem~\ref{thm:multiv-tail} to get
\begin{align}\label{eq:prob-estimate-m-ogp}
    \mathbb{P}\left[\Phi_{\mathcal{I}_i} > k^{p/2} \gamma E_{\mathrm{max}}, \forall \,i \right]
    &\leq 
    \varphi_{\boldsymbol{X}}(\boldsymbol{t}) 
    \prod_{i\leq m} \left<e_i, \Sigma^{-1}\boldsymbol{t} \right>
    \,,\nonumber\\
    &\leq
    \exp_2\left[ 
    -\gamma^2 \log_2 \binom{N}{k} \frac{pm}{1+ 2 m p \nu_2^p} \left(1 + o_N(1)\right)
    \right]
    \,.
\end{align}
Combining this bound with~\eqref{eq:sup-bound-slepian} and recalling that $(\mathcal{I}_1,\dots,\mathcal{I}_m)\in F(m,\nu_1,\nu_2)$ is arbitrary, we complete the proof.
\subsection{Proof of Theorem~\ref{thm:IGPT}}\label{IGPT-Proof}
We record a classical fact, reproduced from~\cite[Lemma~3.1]{gamarnik2018finding}.
\begin{lemma}\label{Lemma:Gauss-Max}
    Let $Z_1,\dots,Z_n\sim \cN(0,1)$ be i.i.d.\,random variables. There exists an $N\in\mathbb{N}$ and a constant $c>0$ such that for all $n>N$,
    \[
    \mathbb{P}\left[\sqrt{2\log n}\left(\max_{1\le i\le n}Z_i -b_n\right)\le \frac32\log\log n\right]\ge 1-\frac{c}{\log^{1.5}n},
    \]
    where
    \[
    b_n :=\sqrt{2\log n} - \frac{\log (4\pi \log n)}{2\sqrt{2\log n}}.
    \]
\end{lemma}
In what follows, we use $N/k$ in place of $\lfloor N/k\rfloor$ for simplicity. Let $\mathcal{I}\mathcal{G}\mathcal{P}\text{-T}(\boldsymbol{A}) = (I_1,\dots,I_p)$ where $I_j = \{i_1^{(j)},\dots,i_k^{(j)}\}\subset [N]$ for $1\le j\le p$. Define the sets
\[
I_{u,t}:=\left\{i_1^{(u)},\dots,i_t^{(u)}\right\},\quad 1\le u\le p,\,\,1\le t\le k.
\]
By convention, $I_u = I_{u,k}$ for $1\le u\le p$. 
\paragraph{Initializaition} Note that $A_{i_1^{(1)},\dots,i_p^{(1)}} = \max_{j\in P_1} A_{i_1^{(1)},\dots,i_{p-1}^{(1)},j}$, where $|P_1|=N/k$. Using Lemma~\ref{Lemma:Gauss-Max}, we obtain that
\begin{align}\label{eq:Initialize}
&\mathbb{P}\Bigl[E_{\mathrm{Initial}}\Bigr]\ge 1-\frac{c}{\log^{1.5}(N/k)},\quad\text{where}\\
&E_{\mathrm{Initial}}:=\left\{\left|\sqrt{2\log(N/k)}\left(A_{i_1^{(1)},\dots,i_p^{(1)}} - b_{N/k}\right)\right|\le \frac32\log\log \frac{N}{k}\right\}.  \nonumber  
\end{align}
\paragraph{Loops of \IGPT} We formalize ``Loop-2" of Algorithm~\ref{IGPT}. Set
\[
M_{j,t}:=\max_{i_j\in P_t} \sum_{\substack{\displaystyle  i_1\in I_{1,t},\dots,i_{j-1}\in I_{j-1,t} \\ \displaystyle i_{j+1}\in I_{j+1,t-1},\dots,i_p\in I_{p,t-1}}}A_{\displaystyle i_1,\dots,i_{j-1},i_j,i_{j+1},\dots,i_p}
\]
for $1\le j\le p$ and $2\le t\le k$. Using Lemma~\ref{Lemma:Gauss-Max}, we conclude
\begin{align}\label{eq:E-j-t}
    &\mathbb{P}\bigl[E_{j,t}\bigr]\ge 1-\frac{c}{\log^{1.5}(N/k)},\quad\text{where}\\
   & E_{j,t}:=\left\{\left|\sqrt{2\log(N/k)}\left(\frac{M_{j,t}}{\sqrt{\Pi_{j,t}}} - b_{N/k}\right)\right|\le \frac32\log\log \frac{N}{k}\right\} \quad\text{and}\nonumber\\
   &\Pi_{j,t}=\prod_{1\le u\le j-1}|I_{u,t}|\cdot \prod_{j+1\le u\le p}|I_{u,t-1}|= t^{j-1}(t-1)^{p-j}\nonumber.
\end{align}
Note that the events $E_{\mathrm{Initial}}$ and $E_{j,t}$, $1\le j\le p$, $2\le t\le k$ are mutually independent. Thus,
\begin{align}
  \mathbb{P}\left[E_{\mathrm{Initial}}\cap \bigcap_{\substack{1\le j\le p \\ 2\le t\le k}}E_{j,t}\right]   &= \mathbb{P}[E_{\mathrm{Initial}}] \cdot \prod_{\substack{1\le j\le p \\ 2\le t\le k}} \mathbb{P}[E_{j,t}] \nonumber \\
  &=\left(1-\frac{c}{\log^{1.5}(N/k)}\right)^{1+p(k-1)}\label{eq:Prob-ALG} \\
  &\ge 1-\frac{c(1+p(k-1))}{\log^{1.5}(N/k)}\ge 1-\frac{2cpk}{\log^{1.5} N},\label{eq:BERN-INEQ}
\end{align}
for all large enough $N$ and all $k\le f(N)$,
where~\eqref{eq:Prob-ALG} uses~\eqref{eq:Initialize} and~\eqref{eq:E-j-t}, and~\eqref{eq:BERN-INEQ} uses Bernoulli inequality together with $k=o(\log^{1.5}N)$ so that $\log^{1.5}(N/k) =(1+o(1))\log^{1.5} N$.

For $k=o(\log^{1.5} N)$, the probability guarantee in~\eqref{eq:BERN-INEQ} is $1-o_N(1)$ since $p,c$ are constants. Observing that
\begin{equation}\label{MU}
 \boldsymbol{\mu}:=\frac{1}{k^p}\sum_{i_1\in I_1,\dots,i_p\in I_p}A_{i_1,\dots,i_p} = \frac{1}{k^p}\left(A_{i_1^{(1)},\dots,i_p^{(1)}} + \sum_{1\le j\le p}\sum_{2\le t\le k}M_{j,t}\right),
\end{equation}
 we bound $\boldsymbol{\mu}$ on the event $E_{\mathrm{Initial}}\cap \bigcap_{j,t}E_{j,t}$ which holds w.h.p. To set the stage, notice that for $N$ sufficiently large, we have $2\log(N/k) = 2\log N - 2\log k \ge \log N$ as $k=o(\log^{1.5}N)$,  $b_{N/k}<\sqrt{2\log N}$ and $\log\log(N/k)<\log\log N$. Moreover,~\eqref{eq:E-j-t} yields
\begin{equation}\label{eq:Pi-j}
    (t-1)^{p-1}\le \Pi_{j,t}\le t^{p-1}.
\end{equation}
We are ready to bound $\boldsymbol{\mu}$. We have
\begin{align}
    \boldsymbol{\mu}&\le \frac{1}{k^p}\left(\sqrt{2\log N} +\frac{\frac32\log\log N}{\sqrt{\log N}}\right)\left(1+\sum_{\substack{1\le j\le p\\ 2\le t\le k}}\sqrt{\Pi_{j,t}}\right)\label{use:eq:E-j-t} \\
    &\le \frac{1}{k^p}\left(\sqrt{2\log N} +\frac{\frac32\log\log N}{\sqrt{\log N}}\right)\left(1+p\sum_{2\le t\le k}t^{\frac{p-1}{2}}\right)\label{eq:Crudely}\\
    &\le \frac{1}{k^p}\left(\sqrt{2\log N} +\frac{\frac32\log\log N}{\sqrt{\log N}}\right)\left(1+pk^{\frac{p-1}{2}} + \frac{2p}{p+1}\left(k^{\frac{p+1}{2}} - 2^{\frac{p+1}{2}}\right)\right)\label{eq:INT-ESTIMATE} \\
    &=\frac{2\sqrt{p}}{p+1}\sqrt{\frac{2p\log N}{k^{p-1}}}\left(1+o_{k,N}(1)\right) = \frac{2\sqrt{p}}{p+1}E_{\mathrm{max}}\left(1+o_{k,N}(1)\right).\label{eq:LAST}
\end{align}
We now justify the lines~\eqref{use:eq:E-j-t}-\eqref{eq:LAST}. Line~\eqref{use:eq:E-j-t} uses~\eqref{eq:Initialize},~\eqref{eq:E-j-t}, and~\eqref{MU}; and line~\eqref{eq:Crudely} uses~\eqref{eq:Pi-j}. 
As for the line~\eqref{eq:INT-ESTIMATE}, notice the integral estimate
\[
\sum_{2\le t\le k}t^{\frac{p-1}{2}} \le k^{\frac{p-1}{2}} + \int_2^k t^{\frac{p-1}{2}}\;dt = k^{\frac{p-1}{2}}+\frac{2}{p+1}\left(k^{\frac{p+1}{2}}-2^{\frac{p+1}{2}}\right).
\]
Lastly, for~\eqref{eq:LAST} recall the facts $p=O(1)$, $k,N=\omega(1)$, as well as the fact that for $k=o(\log^{1.5}N)$, under which
\[
\log\binom{N}{k} = k\log \frac{N}{k}(1+o_{N,k}(1)) = k\log N \left(1+o_{N,k}(1)\right),
\]
so that~\eqref{eq:E-star} becomes
\[
E_{\mathrm{max}} = \sqrt{\frac{2p}{k^p}\log\binom{N}{k}} = \sqrt{\frac{2p \log N}{k^{p-1}}}\left(1+o_{k,N}(1)\right).
\]
The lower bound for $\boldsymbol{\mu}$ is similar and omitted. This completes the proof of Theorem~\ref{thm:IGPT}.
\subsubsection*{Acknowledgments} We thank Partha S.\,Dey and David Gamarnik for valuable discussions.
\bibliographystyle{amsalpha}
\bibliography{ref}

\newcommand{\etalchar}[1]{$^{#1}$}
\providecommand{\bysame}{\leavevmode\hbox to3em{\hrulefill}\thinspace}
\providecommand{\MR}{\relax\ifhmode\unskip\space\fi MR }
\providecommand{\MRhref}[2]{%
  \href{http://www.ams.org/mathscinet-getitem?mr=#1}{#2}
}
\providecommand{\href}[2]{#2}
\begin{thebibliography}{EMKPOZ25}

\bibitem[AA24]{albanese2024boolean}
Linda Albanese and Andrea Alessandrelli, \emph{{Boolean SK Model}}, arXiv preprint arXiv:2409.08693 (2024).

\bibitem[ABM20]{addario2020algorithmic}
Louigi Addario-Berry and Pascal Maillard, \emph{The algorithmic hardness threshold for continuous random energy models}, Mathematical Statistics and Learning \textbf{2} (2020), no.~1, 77--101.

\bibitem[AC14]{auffinger2014free}
Antonio Auffinger and Wei-Kuo Chen, \emph{Free energy and complexity of spherical bipartite models}, Journal of Statistical Physics \textbf{157} (2014), no.~1, 40--59.

\bibitem[ACCD11]{arias2011detection}
Ery Arias-Castro, Emmanuel~J Candes, and Arnaud Durand, \emph{Detection of an anomalous cluster in a network}, The Annals of Statistics (2011), 278--304.

\bibitem[AGK24]{anschuetz2024bounds}
Eric~R Anschuetz, David Gamarnik, and Bobak~T Kiani, \emph{Bounds on the ground state energy of quantum $ p $-spin hamiltonians}, arXiv preprint arXiv:2404.07231 (2024).

\bibitem[ALS22a]{abbe2022binary}
Emmanuel Abbe, Shuangping Li, and Allan Sly, \emph{Binary perceptron: efficient algorithms can find solutions in a rare well-connected cluster}, Proceedings of the 54th Annual ACM SIGACT Symposium on Theory of Computing, 2022, pp.~860--873.

\bibitem[ALS22b]{abbe2022proof}
\bysame, \emph{Proof of the contiguity conjecture and lognormal limit for the symmetric perceptron}, 2021 IEEE 62nd Annual Symposium on Foundations of Computer Science (FOCS), IEEE, 2022, pp.~327--338.

\bibitem[AS16]{alon2016probabilistic}
Noga Alon and Joel~H Spencer, \emph{The probabilistic method}, John Wiley \& Sons, 2016.

\bibitem[AT09]{adler2009random}
Robert~J Adler and Jonathan~E Taylor, \emph{Random fields and geometry}, Springer Science \& Business Media, 2009.

\bibitem[AXY24]{auddy2024tensor}
Arnab Auddy, Dong Xia, and Ming Yuan, \emph{Tensor methods in high dimensional data analysis: Opportunities and challenges}, arXiv preprint arXiv:2405.18412 (2024).

\bibitem[BDN17]{bhamidi2017energy}
Shankar Bhamidi, Partha~S Dey, and Andrew~B Nobel, \emph{Energy landscape for large average submatrix detection problems in gaussian random matrices}, Probability Theory and Related Fields \textbf{168} (2017), 919--983.

\bibitem[BGT10]{bayati2010combinatorial}
Mohsen Bayati, David Gamarnik, and Prasad Tetali, \emph{Combinatorial approach to the interpolation method and scaling limits in sparse random graphs}, Proceedings of the forty-second ACM symposium on Theory of computing, 2010, pp.~105--114.

\bibitem[BH22]{bresler2022algorithmic}
Guy Bresler and Brice Huang, \emph{The algorithmic phase transition of random k-sat for low degree polynomials}, 2021 IEEE 62nd Annual Symposium on Foundations of Computer Science (FOCS), IEEE, 2022, pp.~298--309.

\bibitem[BI13]{butucea-ingster}
Cristina Butucea and Yuri~I. Ingster, \emph{{Detection of a sparse submatrix of a high-dimensional noisy matrix}}, Bernoulli \textbf{19} (2013), no.~5B, 2652 -- 2688.

\bibitem[BKR{\etalchar{+}}11]{balakrishnan2011statistical}
Sivaraman Balakrishnan, Mladen Kolar, Alessandro Rinaldo, Aarti Singh, and Larry Wasserman, \emph{Statistical and computational tradeoffs in biclustering}, NIPS 2011 workshop on computational trade-offs in statistical learning, vol.~4, 2011.

\bibitem[Bor75]{borell1975brunn}
Christer Borell, \emph{The brunn-minkowski inequality in gauss space}, Inventiones mathematicae \textbf{30} (1975), no.~2, 207--216.

\bibitem[BPW18]{bandeira2018notes}
Afonso~S Bandeira, Amelia Perry, and Alexander~S Wein, \emph{Notes on computational-to-statistical gaps: predictions using statistical physics}, Portugaliae mathematica \textbf{75} (2018), no.~2, 159--186.

\bibitem[BTY{\etalchar{+}}21]{bi2021tensors}
Xuan Bi, Xiwei Tang, Yubai Yuan, Yanqing Zhang, and Annie Qu, \emph{Tensors in statistics}, Annual review of statistics and its application \textbf{8} (2021), no.~1, 345--368.

\bibitem[CMP{\etalchar{+}}23]{charbonneau2023spin}
Patrick Charbonneau, Enzo Marinari, Giorgio Parisi, Federico Ricci-tersenghi, Gabriele Sicuro, Francesco Zamponi, and Marc Mezard, \emph{Spin glass theory and far beyond: Replica symmetry breaking after 40 years}, World Scientific, 2023.

\bibitem[CX16]{chen2016statistical}
Yudong Chen and Jiaming Xu, \emph{Statistical-computational tradeoffs in planted problems and submatrix localization with a growing number of clusters and submatrices}, Journal of Machine Learning Research \textbf{17} (2016), no.~27, 1--57.

\bibitem[Der80]{derrida1980random}
Bernard Derrida, \emph{Random-energy model: Limit of a family of disordered models}, Physical Review Letters \textbf{45} (1980), no.~2, 79.

\bibitem[Der81]{derrida1981random}
\bysame, \emph{Random-energy model: An exactly solvable model of disordered systems}, Physical Review B \textbf{24} (1981), no.~5, 2613.

\bibitem[DGH25]{du2025algorithmic}
Hang Du, Shuyang Gong, and Rundong Huang, \emph{The algorithmic phase transition of random graph alignment problem}, Probability Theory and Related Fields (2025), 1--56.

\bibitem[DSS15]{ding2015proof}
Jian Ding, Allan Sly, and Nike Sun, \emph{Proof of the satisfiability conjecture for large k}, Proceedings of the forty-seventh annual ACM symposium on Theory of computing, 2015, pp.~59--68.

\bibitem[EA24]{alaoui2024near}
Ahmed El~Alaoui, \emph{Near-optimal shattering in the ising pure p-spin and rarity of solutions returned by stable algorithms}, arXiv preprint arXiv:2412.03511 (2024).

\bibitem[EAMS23a]{alaoui2023sampling}
Ahmed El~Alaoui, Andrea Montanari, and Mark Sellke, \emph{Sampling from mean-field gibbs measures via diffusion processes}, arXiv preprint arXiv:2310.08912 (2023).

\bibitem[EAMS23b]{alaoui2023shattering}
\bysame, \emph{Shattering in pure spherical spin glasses}, arXiv preprint arXiv:2307.04659 (2023).

\bibitem[ECG22]{cheairi2022densest}
Houssam El~Cheairi and David Gamarnik, \emph{Densest subgraphs of a dense erd{\"o}s-r{\'e}nyi graph. asymptotics, landscape and universality}, arXiv preprint arXiv:2212.03925 (2022).

\bibitem[EKPOZ24]{Erba_2024}
Vittorio Erba, Florent Krzakala, Rodrigo Pérez~Ortiz, and Lenka Zdeborová, \emph{Statistical mechanics of the maximum-average submatrix problem}, Journal of Statistical Mechanics: Theory and Experiment \textbf{2024} (2024), no.~1, 013403.

\bibitem[EMKPOZ25]{MAS-2025}
Vittorio Erba, Nathan Malo~Kupferschmid, Rodrigo Pérez~Ortiz, and Lenka Zdeborová, \emph{The maximum-average subtensor problem: equilibrium and out-of-equilibrium properties}, arXiv preprint arXiv:2506.15400 (2025).

\bibitem[For10]{fortunato2010community}
Santo Fortunato, \emph{Community detection in graphs}, Physics reports \textbf{486} (2010), no.~3-5, 75--174.

\bibitem[Fri90]{frieze1990independence}
Alan~M Frieze, \emph{On the independence number of random graphs}, Discrete Mathematics \textbf{81} (1990), no.~2, 171--175.

\bibitem[Gam21]{Gamarnik_2021}
David Gamarnik, \emph{The overlap gap property: A topological barrier to optimizing over random structures}, Proceedings of the National Academy of Sciences \textbf{118} (2021), no.~41.

\bibitem[Gam25]{gamarnik2025turing}
\bysame, \emph{Turing in the shadows of nobel and abel: an algorithmic story behind two recent prizes}, arXiv preprint arXiv:2501.15312 (2025).

\bibitem[GG25]{BranchingOGPSharp}
David Gamarnik and Shuyang Gong, In preparation (2025).

\bibitem[GJ21]{GamarnikJagannathAMP}
David Gamarnik and Aukosh Jagannath, \emph{{The overlap gap property and approximate message passing algorithms for $p$-spin models}}, The Annals of Probability \textbf{49} (2021), no.~1, 180 -- 205.

\bibitem[GJK23a]{gamarnik2023shattering}
David Gamarnik, Aukosh Jagannath, and Eren~C K{\i}z{\i}lda{\u{g}}, \emph{Shattering in the ising pure $ p $-spin model}, arXiv preprint arXiv:2307.07461 (2023).

\bibitem[GJK23b]{gamarnik2023shatteringisingpurepspin}
David Gamarnik, Aukosh Jagannath, and Eren~C. Kızıldağ, \emph{Shattering in the ising pure $p$-spin model}, 2023.

\bibitem[GJW20]{gamarnik2020low}
David Gamarnik, Aukosh Jagannath, and Alexander~S Wein, \emph{Low-degree hardness of random optimization problems}, 2020 IEEE 61st Annual Symposium on Foundations of Computer Science (FOCS), IEEE, 2020, pp.~131--140.

\bibitem[GJW24]{gamarnik2024hardness}
\bysame, \emph{Hardness of random optimization problems for boolean circuits, low-degree polynomials, and langevin dynamics}, SIAM Journal on Computing \textbf{53} (2024), no.~1, 1--46.

\bibitem[GK23]{gamarnik2023algorithmic}
David Gamarnik and Eren~C K{\i}z{\i}lda{\u{g}}, \emph{Algorithmic obstructions in the random number partitioning problem}, The Annals of Applied Probability \textbf{33} (2023), no.~6B, 5497--5563.

\bibitem[GKPX22]{gamarnik2022algorithms}
David Gamarnik, Eren~C K{\i}z{\i}lda{\u{g}}, Will Perkins, and Changji Xu, \emph{Algorithms and barriers in the symmetric binary perceptron model}, 2022 IEEE 63rd Annual Symposium on Foundations of Computer Science (FOCS), IEEE, 2022, pp.~576--587.

\bibitem[GKPX23]{gamarnik2023geometric}
David Gamarnik, Eren~C Kizilda{\u{g}}, Will Perkins, and Changji Xu, \emph{Geometric barriers for stable and online algorithms for discrepancy minimization}, The Thirty Sixth Annual Conference on Learning Theory, PMLR, 2023, pp.~3231--3263.

\bibitem[GKW25]{gamarnik2025optimal}
David Gamarnik, Eren~C K{\i}z{\i}lda{\u{g}}, and Lutz Warnke, \emph{Optimal hardness of online algorithms for large independent sets}, arXiv preprint arXiv:2504.11450 (2025).

\bibitem[GL18]{gamarnik2018finding}
David Gamarnik and Quan Li, \emph{{Finding a large submatrix of a Gaussian random matrix}}, The Annals of Statistics \textbf{46} (2018), no.~6A, 2511 -- 2561.

\bibitem[GMZ22]{gamarnik2022disordered}
David Gamarnik, Cristopher Moore, and Lenka Zdeborov{\'a}, \emph{Disordered systems insights on computational hardness}, Journal of Statistical Mechanics: Theory and Experiment \textbf{2022} (2022), no.~11, 114015.

\bibitem[GS14]{gamarnik2014limits}
David Gamarnik and Madhu Sudan, \emph{Limits of local algorithms over sparse random graphs}, Proceedings of the 5th conference on Innovations in theoretical computer science, 2014, pp.~369--376.

\bibitem[Has05]{hashorva2005asymptotics}
Enkelejd Hashorva, \emph{Asymptotics and bounds for multivariate gaussian tails}, Journal of theoretical probability \textbf{18} (2005), no.~1, 79--97.

\bibitem[HH03]{hashorva2003multivariate}
Enkelejd Hashorva and J{\"u}rg H{\"u}sler, \emph{On multivariate gaussian tails}, Annals of the Institute of Statistical Mathematics \textbf{55} (2003), 507--522.

\bibitem[HS25a]{huang2025strong}
Brice Huang and Mark Sellke, \emph{Strong low degree hardness for stable local optima in spin glasses}, arXiv preprint arXiv:2501.06427 (2025).

\bibitem[HS25b]{huang2025tight}
\bysame, \emph{Tight lipschitz hardness for optimizing mean field spin glasses}, Communications on Pure and Applied Mathematics \textbf{78} (2025), no.~1, 60--119.

\bibitem[Hua25]{huang2025statistical}
Brice Huang, \emph{Statistical and algorithmic thresholds in spin glasses}, Ph.D. thesis, Massachusetts Institute of Technology, 2025.

\bibitem[HW53]{hoffman1953variation}
A.~J. Hoffman and H.~W. Wielandt, \emph{{The variation of the spectrum of a normal matrix}}, Duke Mathematical Journal \textbf{20} (1953), no.~1, 37 -- 39.

\bibitem[IST76]{ibragimov1976norms}
IA~Ibragimov, VN~Sudakov, and BS~Tsirelson, \emph{Norms of gaussian sample functions}, Proceedings of the third Japan USSR symposium on probability theory, lecture notes in math, vol. 550, 1976, pp.~20--41.

\bibitem[KBRS11]{kolar2011minimax}
Mladen Kolar, Sivaraman Balakrishnan, Alessandro Rinaldo, and Aarti Singh, \emph{Minimax localization of structural information in large noisy matrices}, Advances in Neural Information Processing Systems \textbf{24} (2011).

\bibitem[K{\i}z23]{kizildaug2023sharp}
Eren~C K{\i}z{\i}lda{\u{g}}, \emph{Sharp phase transition for multi overlap gap property in ising $ p $-spin glass and random $ k $-sat models}, arXiv preprint arXiv:2309.09913 (2023).

\bibitem[KNW22]{kunisky2022strong}
Dmitriy Kunisky and Jonathan Niles-Weed, \emph{Strong recovery of geometric planted matchings}, Proceedings of the 2022 Annual ACM-SIAM Symposium on Discrete Algorithms (SODA), SIAM, 2022, pp.~834--876.

\bibitem[KT87]{kirkpatrick1987p}
Theodore~R Kirkpatrick and Devarajan Thirumalai, \emph{p-spin-interaction spin-glass models: Connections with the structural glass problem}, Physical Review B \textbf{36} (1987), no.~10, 5388.

\bibitem[LSZ24]{li2024discrepancy}
Shuangping Li, Tselil Schramm, and Kangjie Zhou, \emph{Discrepancy algorithms for the binary perceptron}, arXiv preprint arXiv:2408.00796 (2024).

\bibitem[LT13]{ledoux2013probability}
Michel Ledoux and Michel Talagrand, \emph{Probability in banach spaces: isoperimetry and processes}, Springer Science \& Business Media, 2013.

\bibitem[McK24]{mckenna2024complexity}
Benjamin McKenna, \emph{Complexity of bipartite spherical spin glasses}, Annales de l'Institut Henri Poincare (B) Probabilites et statistiques, vol.~60, Institut Henri Poincar{\'e}, 2024, pp.~636--657.

\bibitem[MO04]{madeira2004biclustering}
Sara~C Madeira and Arlindo~L Oliveira, \emph{Biclustering algorithms for biological data analysis: a survey}, IEEE/ACM transactions on computational biology and bioinformatics \textbf{1} (2004), no.~1, 24--45.

\bibitem[MPV87]{1987sgtb.book.....M}
Marc M{\'e}zard, Giorgio Parisi, and Miguel~Angel Virasoro, \emph{Spin glass theory and beyond: An introduction to the replica method and its applications}, vol.~9, World Scientific Publishing Company, 1987.

\bibitem[MRZ15]{montanari2015limitation}
Andrea Montanari, Daniel Reichman, and Ofer Zeitouni, \emph{On the limitation of spectral methods: From the gaussian hidden clique problem to rank-one perturbations of gaussian tensors}, Advances in Neural Information Processing Systems \textbf{28} (2015).

\bibitem[MW15]{ma-wu}
Zongming Ma and Yihong Wu, \emph{{Computational barriers in minimax submatrix detection}}, The Annals of Statistics \textbf{43} (2015), no.~3, 1089 -- 1116.

\bibitem[Nob]{BivTail}
Andrew Nobel, \emph{Gaussian comparison lemma}, \url{https://nobel.web.unc.edu/wp-content/uploads/sites/13591/2017/04/GaussianComparison.pdf}.

\bibitem[NS10]{nion2010tensor}
Dimitri Nion and Nicholas~D Sidiropoulos, \emph{Tensor algebra and multidimensional harmonic retrieval in signal processing for mimo radar}, IEEE Transactions on Signal Processing \textbf{58} (2010), no.~11, 5693--5705.

\bibitem[Pan13]{panchenko2013sherrington}
Dmitry Panchenko, \emph{{The {S}herrington-{K}irkpatrick {M}odel}}, Springer Science \& Business Media, 2013.

\bibitem[PX21]{perkins2021frozen}
Will Perkins and Changji Xu, \emph{Frozen 1-rsb structure of the symmetric ising perceptron}, Proceedings of the 53rd Annual ACM SIGACT Symposium on Theory of Computing, 2021, pp.~1579--1588.

\bibitem[Sav62]{savage1962mills}
I~Richard Savage, \emph{Mills’ ratio for multivariate normal distributions}, J. Res. Nat. Bur. Standards Sect. B \textbf{66} (1962), no.~3, 93--96.

\bibitem[Sel25]{sellke2025tight}
Mark Sellke, \emph{Tight low degree hardness for optimizing pure spherical spin glasses}, arXiv preprint arXiv:2504.04632 (2025).

\bibitem[SK75]{sherrington1975solvable}
David Sherrington and Scott Kirkpatrick, \emph{Solvable model of a spin-glass}, Physical review letters \textbf{35} (1975), no.~26, 1792.

\bibitem[Sle62]{slepian1962one}
David Slepian, \emph{The one-sided barrier problem for gaussian noise}, Bell System Technical Journal \textbf{41} (1962), no.~2, 463--501.

\bibitem[SM50]{sherman1950adjustment}
Jack Sherman and Winifred~J. Morrison, \emph{Adjustment of an inverse matrix corresponding to a change in one element of a given matrix}, The Annals of Mathematical Statistics \textbf{21} (1950), no.~1, 124--127.

\bibitem[SN13]{sun2013maximal}
Xing Sun and Andrew~B Nobel, \emph{On the maximal size of large-average and anova-fit submatrices in a gaussian random matrix}, Bernoulli: official journal of the Bernoulli Society for Mathematical Statistics and Probability \textbf{19} (2013), no.~1, 275.

\bibitem[SWPN09]{shabalin2009finding}
Andrey~A Shabalin, Victor~J Weigman, Charles~M Perou, and Andrew~B Nobel, \emph{Finding large average submatrices in high dimensional data}, The Annals of Applied Statistics (2009), 985--1012.

\bibitem[Tal00]{talagrand2000rigorous}
Michel Talagrand, \emph{Rigorous low-temperature results for the mean field p-spins interaction model}, Probability theory and related fields \textbf{117} (2000), 303--360.

\bibitem[Tal05]{talagrand2005generic}
\bysame, \emph{The generic chaining: upper and lower bounds of stochastic processes}, Springer Science \& Business Media, 2005.

\bibitem[Tal06]{talagrand2006parisi}
\bysame, \emph{The parisi formula}, Annals of mathematics (2006), 221--263.

\bibitem[Tal11]{talagrand2011mean}
M.~Talagrand, \emph{Mean field models for spin glasses: Volume ii: Advanced replica-symmetry and low temperature}, Ergebnisse der Mathematik und ihrer Grenzgebiete. 3. Folge / A Series of Modern Surveys in Mathematics, Springer Berlin Heidelberg, 2011.

\bibitem[VH14]{van2014probability}
Ramon Van~Handel, \emph{Probability in high dimension}, Lecture Notes (Princeton University) \textbf{2} (2014), no.~3, 2--3.

\bibitem[Wei22]{wein2020optimal}
Alexander~S Wein, \emph{Optimal low-degree hardness of maximum independent set}, Mathematical Statistics and Learning \textbf{4} (2022), no.~3, 221--251.

\bibitem[WX21]{wu2021statistical}
Yihong Wu and Jiaming Xu, \emph{Statistical problems with planted structures: Information-theoretical and computational limits}, Information-Theoretic Methods in Data Science \textbf{383} (2021), 13.

\bibitem[ZAZD19]{zhang2019tensor}
Zhengwu Zhang, Genevera~I Allen, Hongtu Zhu, and David Dunson, \emph{Tensor network factorizations: Relationships between brain structural connectomes and traits}, Neuroimage \textbf{197} (2019), 330--343.

\bibitem[ZLZ13]{zhou2013tensor}
Hua Zhou, Lexin Li, and Hongtu Zhu, \emph{Tensor regression with applications in neuroimaging data analysis}, Journal of the American Statistical Association \textbf{108} (2013), no.~502, 540--552.

\end{thebibliography}

\end{document}